\documentclass[reqno]{amsart}
\usepackage{amsfonts,amsmath,amssymb, mathtools, amstext}
\usepackage{dsfont}
\numberwithin{equation}{section}

\newtheorem{thm}{Theorem}[section]
\newtheorem*{thm*}{Theorem}
\newtheorem{lem}{Lemma}[section]

\newtheorem{cor}{Corollary}[section]
\newtheorem*{cor*}{Corollary}
\newtheorem{prop}{Proposition}[section]
\newtheorem{defi}{Definition}[section]

\newtheorem{Step}{Step}[section]

\begin{document}
\title[KGMP equations in high dimensions]
{Klein-Gordon-Maxwell equations in high dimensions}
\author{Pierre-Damien Thizy}
\address{Pierre-Damien Thizy, Universit\'e de Cergy-Pontoise, 
D\'epartement de Math\'ematiques, Site de 
Saint-Martin, 2 avenue Adolphe Chauvin, 
95302 Cergy-Pontoise cedex, 
France}
\email{pierre-damien.thizy@u-cergy.fr}

\begin{abstract}
We prove the existence of a mountain-pass solution and the a priori bound property for the electrostatic Klein-Gordon-Maxwell equations in high dimension.
\end{abstract}

\maketitle

 In what follows we let $(M,g)$ be a smooth closed Riemannian $n$-manifold, $n \ge 3$, closed meaning compact without boundary. We let $2^\star=\frac{2n}{n-2}$ be the critical Sobolev exponent for the embeddings of $H^1$, the Sobolev space of functions in $L^2$ with one derivative in $L^2$. 
We let also $p \in (2,2^\star]$, $q>0$, $m_0,m_1>0$, 
and $\omega\in(-m_0,m_0)$ be real numbers. The electrostatic Klein-Gordon-Maxwell system we investigate in this paper is written as
\begin{equation}\label{kg}
\begin{cases}
\Delta_g u+m_0^2 u=u^{p-1}+\omega^2(1-qv)^2 u\\
\Delta_g v+(m_1^2+q^2u^2) v= q u^2~,
\end{cases}
\end{equation}
where $\Delta_g=-\text{div}_g\nabla$ is the Laplace-Beltrami operator. This system arises  
when looking for standing waves solutions of the full Klein-Gordon-Maxwell system in Proca formalism (see Hebey and Truong \cite{kgmp4}). The first equation in \eqref{kg} 
is  energy critical when $p = 2^\star$. The second equation in \eqref{kg} is energy supercritical when $n\geq 5$. The Proca mass 
$m_1>0$ makes that the two equations in \eqref{kg} are strongly coupled one with another. 

\medskip The system \eqref{kg}, in Proca form, has been investigated by Druet and Hebey \cite{kgmp3}, 
Druet, Hebey and V\'etois \cite{kgmp4DHV}, Hebey and Truong \cite{kgmp4}, and Hebey and Wei \cite{HebWei1} in the case 
of $3$ and $4$-dimensional manifolds (see also Hebey and Wei \cite{schrp}). 
In these dimensions, the second equation in \eqref{kg} is either subcritical or critical and we do have a well established variational framework for the system. 
When $n \ge 5$, as already mentioned, the second equation in \eqref{kg} is supercritical. The problem comes with the $u^2v$-term in the left had side of the equation as well as with the $u^2$-term in the right hand side of the equation (there holds that $2+1 > 2^\star-1$ when $n \ge 5$, and we even have that $2 > 2^\star-1$ when $n \ge 7$). A priori 
we lose a variational framework for the system in these dimensions but, as we will see in 
Section \ref{VarStrSecEqt}, such a variational framework can be restored thanks to the very 
special structure of the second equation in \eqref{kg}. Then we can ask the question of the existence of solutions of 
\eqref{kg} with special variational structures, and more precisely the question of the existence of solutions with a mountain-pass structure. 
We investigate this question in this paper, as well as the more involved question of the existence 
of a priori bounds for arbitrary solutions of \eqref{kg}. The first result we prove in this paper is concerned with $p < 2^\star$. The $3$-dimensional case in Theorem 
\ref{SubCritThm} below 
is due to Druet and Hebey \cite{kgmp3}, the $4$-dimensional case to Hebey and Truong \cite{kgmp4}, and the $n \ge 5$ cases are new. As one can check, it  
follows from  the theorem that when 
the first equation in \eqref{kg} is subcritical, then the possible supercriticality of the second equation has no importance. The notion of a smooth positive 
mountain-pass solution of \eqref{kg} is defined in Section \ref{MountainPassSol}. 

\begin{thm}[Subcritical case]\label{SubCritThm}  Let $(M,g)$ be a smooth closed Riemannian $n$-manifold, $n \ge 3$, $m_0, m_1, q>0$ be positive real numbers, and $p \in (2,2^\star)$. For any 
$\omega\in (-m_0,m_0)$, there exists 
a smooth positive mountain-pass solution $(u,v)$ for \eqref{kg}. Moreover, there also exists $C>0$ such that $\|u\|_{C^{2,\theta}}\leq C$ and $\|v\|_{C^{2,\theta}}\leq C$ for any positive solution 
$(u,v)$ of \eqref{kg} and all $\omega\in (-m_0,m_0)$.
\end{thm}

In the critical case of the first equation in \eqref{kg} there holds that $p = 2^\star$.  Then, as shown in Hebey and Wei \cite{schrp}, and Druet, Hebey and V\'etois \cite{kgmp4DHV}, 
resonant states appear in particular situations, and we cannot get a priori bounds for all phases as in Theorem \ref{SubCritThm}. Our second result establishes that 
the conclusions of Theorem \ref{SubCritThm} are still valid when the potential $m_0^2$ in the first equation of \eqref{kg} is geometrically small, despite the supercriticality of the 
second equation in high dimensions. The $3$-dimensional case in 
Theorem \ref{CritThm} is due to Druet and Hebey \cite{kgmp3}, the $4$-dimensional case to Hebey and Truong \cite{kgmp4}, and the $n \ge 5$ cases, as for Theorem \ref{SubCritThm}, are new. 

\begin{thm}[Critical case]\label{CritThm}  Let $(M,g)$ be a smooth closed Riemannian $n$-manifold, $n \ge 3$, 
$m_0, m_1, q > 0$ be positive real numbers, and $p = 2^\star$. Assume that 
 \begin{equation}\label{prel1}
 m_0^2 < \frac{n-2}{4(n-1)} S_g
 \end{equation}
somewhere in $M$, where $S_g$ is the scalar curvature of $g$. Then, for any 
$\omega\in (-m_0,m_0)$, there exists 
a smooth positive mountain-pass solution $(u,v)$ for \eqref{kg}. If we assume that \eqref{prel1} holds true everywhere in $M$, then 
there also exists $C>0$ such that $\|u\|_{C^{2,\theta}}\leq C$ and $\|v\|_{C^{2,\theta}}\leq C$ for any positive solution 
$(u,v)$ of \eqref{kg} and all $\omega\in (-m_0,m_0)$.
\end{thm}

In the process of the paper we also prove that the phase compensation phenomenon, established in Druet and Hebey \cite{kgmp3} when $n = 3$, and 
Hebey and Truong \cite{kgmp4} when $n = 4$, stops to hold true when $n \ge 5$ (see Corollary \ref{LossPhaseComp}). In addition, we establish that the gauge potential 
$v$ in \eqref{kg} cannot be controlled in H\"older spaces 
$C^{0,\theta}$ if we do not get a similar control on $u$ (see Corollary \ref{NonC0Conv}). We also discuss a model case where 
\eqref{prel1} is not satisfied in Proposition \ref{necess} but the a priori bound property remains valid.

\medskip The above two theorems prevent the full KGMP system from having standing waves solutions with arbitrarily large amplitude. Better, they claim the compactness in the $C^2$-topology of the set of non-negative solutions of \eqref{kg}, as the phase $\omega$ varies in $(-m_0,m_0)$.

\bigskip\noindent \textbf{Acknowledgements.} The author warmly thanks Emmanuel Hebey for his constant support and valuable remarks at every stage of this work, and Bruno Premoselli for many helpful discussions and comments.

\section{A variational setting for the second equation}\label{VarStrSecEqt}

Following a very nice idea  due to Benci and Fortunato \cite{BenciFortunato}, we introduce, from the formal point of view, the auxiliary functional $\Phi$ given by
 \begin{equation}\label{eqphi11}
 \Delta_g \Phi(u) + (m_1^2+q^2 u^2) \Phi(u) = qu^2~.
 \end{equation}
The definition of $\Phi: H^1 \to H^1$ makes sense when $n=3,4$, as shown in  Druet and Hebey \cite{kgmp3} and 
Hebey and Truong \cite{kgmp4}. Moreover, there holds in these dimensions that
\begin{equation}\label{eqphi12}
 0\leq \Phi(u) \leq \frac{1}{q}, \quad \text{for all }u\in H^1.
 \end{equation}
Now we aim to give a meaning to equation \eqref{eqphi11} when $n\geq 5$ and to define its solution $\Phi(u)$ in some suitable sense for all $u\in H^1$. As already mentioned, \eqref{eqphi11} is a priori not variational anymore in $H^1$ when $n \ge 5$ because of the cubic term $u^2 \Phi(u)$. Adapting the ideas in Hebey \cite{solitary} to the closed setting, we prove that we can give a meaning to $\Phi$ when $n\geq 5$ which, as shown in Section \ref{DiffAuxilEnergy}, will be sufficient to get a variational characterisation of the first equation in \eqref{kg}.
  
 \begin{defi}\label{DefWS}
 Let $u\in H^1$ be given. A function $\Phi(u)\in L^\infty\cap H^1$ is said to be a solution of \eqref{eqphi11} in the \textbf{restricted weak sense} if
 \begin{equation}\label{eqphi125}
 \int_M \langle\nabla \Phi, \nabla \varphi \rangle_{g} dv_g+ \int_M (m_1^2+q^2u^2) \Phi \varphi dv_g=q \int_M u^2 \varphi dv_g,
 \end{equation}
 for all $\varphi\in H^1\cap L^\infty$.
\end{defi} 
 
\noindent When $n=3, 4$, we can define $\Phi$ as a true variational solution of \eqref{eqphi11}. It is then locally Lipschitz, differentiable, 
and its differential $D\Phi(u) = V_u$ at $u \in H^1$ is given as the unique solution of
$$\Delta_gV_u(\varphi) + \left(m_1^2 + q^2u^2\right)V_u(\varphi) = 2qu\left(1-q\Phi(u)\right)\varphi$$
for all $\varphi \in H^1$. 
When $n\ge 5$, with the notion of weak solution given in Definition \ref{DefWS}, we can prove that the following result holds true.

 \begin{lem}\label{l1}
 Let $(M,g)$ be a closed Riemannian $n$-manifold, $n\geq 5$, $q$, $m_1>0$. There exists $\Phi : H^1\to H^1$ such that 
 \begin{equation}\label{eqphi}
 \Delta_g \Phi(u) + (m_1^2+q^2 u^2) \Phi(u) = qu^2
 \end{equation}
 in the restricted weak sense, and $0\leq \Phi(u) \leq \frac{1}{q}$ 
 for all $u\in H^1$. Moreover, $\Phi$ is locally H\"olderian continuous with coefficient $\theta_n=\min(\frac{n-2}{2(n-4)},1)$.
 \end{lem}
 
\begin{proof}[Proof of Lemma \ref{l1}.] Let $u\in H^1$ be given. Let $\Lambda>0$ and $u_\Lambda=\min(|u|,\Lambda)$. Then $u_\Lambda \in L^\infty$ and the equation
 \begin{equation}\label{eq2}
 \Delta_g \Phi_\Lambda(u)+(m_1^2+q^2u_\Lambda^2)\Phi_\Lambda(u)=q u_\Lambda^2 
 \end{equation}
 has one and only one solution $\Phi_\Lambda(u)\in H^1$. By the maximum principle, $\Phi_\Lambda(u)\geq 0$ and writing 
 $$\Delta_g\left(\frac{1}{q}-\Phi_\Lambda(u)\right)+(m_1^2+q^2u_\Lambda^2)\left(\frac{1}{q}-\Phi_\Lambda(u)\right)=\frac{m_1^2}{q}>0, $$
 we get that $\Phi_\Lambda(u)\leq \frac{1}{q}$. We take now $(\Lambda_p)_{p}$ an increasing sequence of positive real numbers such that $\Lambda_p\to +\infty$ as $p\to +\infty$. 
 For $p$, $t\in\mathbb{N}$,
 \begin{equation}\label{AddedEqt1LemVar}
 \begin{split}
& \Delta_g(\Phi_{\Lambda_p}(u)-\Phi_{\Lambda_t}(u))+(m_1^2+q^2u_{\Lambda_p}^2)(\Phi_{\Lambda_p}(u)-\Phi_{\Lambda_t}(u))\\
&= q(u_{\Lambda_p}^2-u_{\Lambda_t}^2)(1-q\Phi_{\Lambda_t}(u))~.
\end{split}
\end{equation}
We multiply \eqref{AddedEqt1LemVar} by $\Phi_{\Lambda_p}(u)-\Phi_{\Lambda_t}(u)\in H^1$ and integrate over $M$. We assume first $n\geq 6$. Using a H\"older inequality, 
and the Sobolev inequality, we get 
the existence of $C>0$, independent of $p$ and $t$, such that
\begin{equation*}
\begin{split}
& \min(m_1^2,1) \|\Phi_{\Lambda_p}(u)-\Phi_{\Lambda_t}(u)\|_{H^1}^2\\
&\leq q \int_M |u_{\Lambda_p}^2-u_{\Lambda_t}^2||\Phi_{\Lambda_p}(u)-\Phi_{\Lambda_t}(u)| dv_g\\
&\leq C \|u_{\Lambda_p}-u_{ \Lambda_t}\|_{L^{2^\star}}\|u\|_{L^{2^\star}} \|\Phi_{\Lambda_p}(u)-\Phi_{ \Lambda_t}(u)\|_{L^{2^\star/(2^\star-2)}}\\
&\leq C \|u\|_{H^1} \|u_{\Lambda_p}-u_{ \Lambda_t}\|_{H^1} \|\Phi_{\Lambda_p}(u)-\Phi_{ \Lambda_t}(u)\|_{H^1}^{2^\star-2}~,
 \end{split}
 \end{equation*}
 where we have used that $2^\star/(2^\star-2) \ge 2^\star$ when $n \ge 6$. If we assume $n = 5$, then
\begin{eqnarray}
\nonumber  \|\Phi_{\Lambda_p}(u)-\Phi_{ \Lambda_t}(u)\|_{H^1}^2
&\leq & C \|u\|_{H^1} \|u_{\Lambda_p}-u_{ \Lambda_t}\|_{H^1} \|\Phi_{\Lambda_p}(u)-\Phi_{ \Lambda_t}(u)\|_{H^1},
\end{eqnarray} 
using H\"older inequality, the Sobolev embedding theorem and that $M$ has finite volume.
  In any case, we get that $(\Phi_{\Lambda_p}(u))_p$ is a Cauchy sequence in $H^1$. Hence, there exists $\Phi=\Phi(u) \in H^1$ such that 
 \begin{equation}\label{eq1}
 \Phi_{\Lambda_p}(u)\to \Phi\quad  \text{in} \quad  H^1\cap L^q
 \end{equation}
 as $p\to +\infty$ for all $q \ge 1$. Up to a subsequence, we may assume that $\Phi_{\Lambda_p}\to \Phi$ a.e. In particular, $0\leq \Phi\leq \frac{1}{q}$ and by \eqref{eq2} and \eqref{eq1}, we get that $\Phi(u)$ satisfies \eqref{eqphi} in the restricted weak sense. There holds that $\Phi(u)$ is unique in $H^1\cap L^\infty$. Testing the equation satisfied by $\Phi(u)-\Phi(v)$ in the restricted weak sense against $\Phi(u)-\Phi(v)\in H^1\cap L^\infty$, we get the same estimates as above, namely
 $$\|\Phi(u)-\Phi(v)\|_{H^1}^{1/\theta_n}\leq C (\|u\|_{H^1}+\|v\|_{H^1})\|u-v\|_{H^1}$$
 for all $u, v \in H^1$, where $\theta_n=1$ if $n=5$ and $\theta_n=1/(4-2^\star)$ if $n\geq6$. This proves Lemma \ref{l1}.
 \end{proof}
 
 When $n=5$, though the problem is not variational in $H^1$, we can prove that 
$\Phi(u)$ is actually a true weak solution in $H^1$ of \eqref{eqphi}. Independently, whatever the dimension is, multiplying by $\Phi(v)-\Phi(u)$ 
the equation satisfied by $\Phi(v)-\Phi(u)$, and since $0 \le \Phi \le \frac{1}{q}$, we get that 
\begin{equation}\label{EstPhiHolderType}
 \|\Phi(u)-\Phi(v)\|_{H^1}^2\leq C \|u+v\|_{L^2} \|v-u\|_{L^2}
 \end{equation}
 for all $u, v \in H^1$, where $C > 0$ is independent of $u$ and $v$.

 \section{Non smooth convergence of the gauge and loss of phase compensation in the critical case}\label{NSCLPC}
 
 We prove in this section that, in the model case of a bubble, the associated gauge potentials are not controlled and we loose the key equation from which phase 
 compensation was established in Druet and Hebey \cite{kgmp3} when $n = 3$, and Hebey and Truong \cite{kgmp4} when $n = 4$. Given a converging sequence 
 $(x_\alpha)_\alpha$ of points in $M$, and a sequence $(\mu_\alpha)_\alpha$ of positive real numbers such that $\mu_\alpha \to 0$ as $\alpha \to +\infty$, we define 
 the bubble of centers $x_\alpha$ and weights $\mu_\alpha$ as the sequence $(B_\alpha)_\alpha$ of functions given by
 \begin{equation}\label{bulle}
 B_\alpha(x) = \left(\frac{\mu_\alpha}{\mu_\alpha^2 + \frac{d_g(x_\alpha,x)^2}{n(n-2)}}\right)^{\frac{n-2}{2}}
 \end{equation}
 for all $\alpha$, and all $x \in M$, where $d_g$ is the distance associated to $g$. Bubbles are constructed from the Caffarelli, Gidas and Spruck \cite{CaffarelliGidas} classification of 
 nonnegative nontrivial solutions of the critical Euclidean equation $\Delta u = u^{2^\star-1}$. Any $C^2$ nonnegative nontrivial solution of this equation is indeed, 
 up to translations and scaling,  given by $B(x) = \left(1+\lambda_n\vert x\vert^2\right)^{-(n-2)/2}$ for all $x\in M$, where $\lambda_n = 1/n(n-2)$. 
Obviously, there holds that
 \begin{equation}\label{RescBubbleCv}
 \mu_\alpha^{\frac{n-2}{2}}B_\alpha\left(\exp_{x_\alpha}(\mu_\alpha x)\right) = B(x)
 \end{equation}
 for all $x \in \mathbb{R}^n$ and all $\alpha \gg 1$. The first result we prove in this section is the following.
 
\begin{lem}\label{controlv2}
Let $(M,g)$ be a closed Riemannian $n$-manifold, $n\geq 5$, $(u_\alpha)_\alpha$ and $(v_\alpha)_\alpha$ be sequences of smooth positive functions in $M$ such that 
$v_\alpha=\Phi(u_\alpha)$, where $\Phi$ is as in \eqref{eqphi}. Let $(\mu_\alpha)_\alpha$ be a sequence of positive real numbers converging to zero, 
and $(x_\alpha)_\alpha$ be a converging sequence of points in $M$. Assume 
\begin{equation} \label{eql612}
 \mu_\alpha^{\frac{n-2}{2}} u_\alpha\left(\exp_{x_\alpha}(\mu_\alpha x)\right) \to \tilde{u}_0(x)
\end{equation} 
in $C^0_{loc}(\mathbb{R}^n)$ as $\alpha\to +\infty$, where $\tilde{u}_0$ is a given positive $C^1$-function in $\mathbb{R}^n$. Then
  \begin{equation}\label{eqcontrolv2}
 \hat{v}_\alpha\to \frac{1}{q}~\text{in}~L^p_{loc}(\mathbb{R}^n)
 \end{equation}
 for all $p \in [1,+\infty)$, and a.e., as $\alpha \to +\infty$, where $\hat{v}_\alpha$ is the function given by $\hat{v}_\alpha(x)=v_\alpha(\exp_{x_\alpha}(\mu_\alpha x))$ 
 for $x \in \mathbb{R}^n$.
\end{lem}  

\begin{proof}[Proof of Lemma \ref{controlv2}] By Lemma \ref{l1}, $0\leq v_\alpha=\Phi(u_\alpha) \leq \frac{1}{q}$ and we only need to prove the almost everywhere convergence in \eqref{eqcontrolv2}. 
Let $w_\alpha$ be given by $w_\alpha=\frac{1}{q}-v_\alpha$. As for the $v_\alpha$'s, we have $0\leq w_\alpha\leq \frac{1}{q}$ over $M$ for all $\alpha$. By \eqref{eqphi}, $w_\alpha$ satisfies that
 $$\Delta_g w_\alpha+(m_1^2+q^2 u_\alpha^2)w_\alpha=\frac{m_1^2}{q}. $$
Let $\delta > 0$ be small. We set for $x\in B_0(\delta/\mu_\alpha)\subset \mathbb{R}^n$,
\begin{equation}\label{defhat}
\begin{split}
&\hat{w}_\alpha(x)=w_\alpha\left(\exp_{x_\alpha}(\mu_\alpha x)\right),\\
&\hat{u}_\alpha(x)=u_\alpha\left(\exp_{x_\alpha}(\mu_\alpha x)\right),\\
&\hat{g}_\alpha(x)=\left(\exp^\star_{x_\alpha}g\right)(\mu_\alpha x),
\end{split}
\end{equation}
so that  $\hat{g}_\alpha\to \xi$ ($\xi$ the Euclidean metric in $\mathbb{R}^n$) in $C^2_{loc}(\mathbb{R}^n)$, as $\alpha\to +\infty$ and 
\begin{equation}\label{eq4}
\Delta_{\hat{g}_\alpha} \hat{w}_\alpha+(m_1^2+q^2 \hat{u}_\alpha^2) \mu_\alpha^2 \hat{w}_\alpha=\frac{m_1^2}{q}\mu_\alpha^2
\end{equation}
in $B_0(\delta/\mu_\alpha)$. Let now $R>0$ be given. By the maximum principle, using  $\hat{w}_\alpha\leq \frac{1}{q}$, \eqref{eql612}, and \eqref{eq4}, we get that for $\alpha$ large enough
  \begin{equation}\label{eq3}
 0\leq\hat{w}_\alpha\leq \bar{w}_\alpha,
 \end{equation}
  where $\bar{w}_\alpha$ is the solution of 
 \begin{equation*}
 \left\{
\begin{array}{lll}
\Delta_{\hat{g}_\alpha}\bar{w}_\alpha+\frac{C_R}{\mu_\alpha^{n-4}}\bar{w}_\alpha & =\frac{m_1^2}{q}\mu_\alpha^2 &\text{ in }B_0(R),\\
\bar{w}_\alpha & =\frac{1}{q} &\text{ on } \partial B_0(R),
\end{array}
\right.
\end{equation*}
 where we can choose
 $$C_R=\frac{q^2}{2}\inf_{B_0(R)}\tilde{u}_0^2>0~.$$
 Since $R$ is any positive constant, it follows from \eqref{eq3} that Lemma \ref{controlv2} is proved if we get that
  \begin{equation}\label{eq5}
 \bar{w}_\alpha \to 0~\text{a.e. in}~B_0(R/2)
 \end{equation}
as $\alpha \to +\infty$. To get this result, we decompose the $\bar{w}_\alpha$'s in a quasi-harmonic part with nonzero Dirichlet boundary condition, and a quasi-Poisson part with zero  Dirichlet boundary condition. More precisely, we write $\bar{w}_\alpha=\bar{w}_{1,\alpha}+\bar{w}_{2,\alpha}$, where
  \begin{equation}\label{eq6}
 \left\{
\begin{array}{lll}
\Delta_{\hat{g}_\alpha}\bar{w}_{1,\alpha}+\frac{C_R}{\mu_\alpha^{n-4}}\bar{w}_{1,\alpha} & =0 &\text{ in }B_0(R)\\
\bar{w}_{1,\alpha} & =\frac{1}{q} &\text{ on } \partial B_0(R),
\end{array}
\right.
\end{equation}
and
\begin{equation}\label{eq7}
 \left\{
\begin{array}{lll}
\Delta_{\hat{g}_\alpha}\bar{w}_{2,\alpha}+\frac{C_R}{\mu_\alpha^{n-4}}\bar{w}_{2,\alpha} & =\frac{m_1^2}{q}\mu_\alpha^2 &\text{ in }B_0(R)\\
\bar{w}_{2,\alpha} & =0 &\text{ on } \partial B_0(R).
\end{array}
\right.
\end{equation}
Multiplying \eqref{eq7} by $\bar{w}_{2,\alpha}$, and integrating over $B_0(R)$, we get that $\bar{w}_{2,\alpha}\to0$ in $H^1(B_0(R))$. Then, by elliptic theory, as developed in Gilbarg and 
Tr\"udinger \cite{Gilbarg}, we get that
 \begin{equation}\label{eq8}
  \bar{w}_{2,\alpha} \to 0 \text{ in } C^0\left(B_0\left(R/2\right)\right).
 \end{equation}
By the maximum principle, we also have that
\begin{equation}\label{eq10}
0\leq \bar{w}_{1,\alpha}\leq \frac{1}{q}~\text{in}~B_0(R)
\end{equation} 
for all $\alpha$.  Since $n\geq5$, $\mu_\alpha^{4-n}\to +\infty$ as $\alpha\to +\infty$, and, up to a subsequence, we can assume that the sequence $(\mu_\alpha^{4-n})_\alpha$ is increasing. 
Then, by the maximum principle and \eqref{eq6}, for any $x\in B_0(R)$, the sequence $(\bar{w}_{1,\alpha}(x))_\alpha$ decreases. In particular, it converges 
to a limit $\bar{w}_1(x)$, 
with $0\leq \bar{w}_1(x)\leq 1/q$. Moreover, we get from \eqref{eq6} that if $\varphi$ a smooth function with compact support in $B_0(R)$, there holds :
\begin{equation*}
\int_{B_0(R)}\left(\Delta_{\hat{g}_\alpha}\varphi +\frac{C_R}{\mu_\alpha^{n-4}}\varphi\right)\bar{w}_{1,\alpha}dv_{\hat{g}_\alpha}=0
\end{equation*}
and then, using the dominated convergence theorem for $\alpha\to +\infty$ and \eqref{eq10},
\begin{equation*}
\mu_\alpha^{4-n}\int_{B_0(R)}\varphi \bar{w}_{1,\alpha} dv_{\hat{g}_\alpha}=\mu_\alpha^{4-n}\left(\int_{B_0(R)} \varphi \bar{w}_1 dv_\xi+o(1)\right)=O(1).
\end{equation*}
 As a conclusion, since $n\geq 5$ and $\mu_\alpha\to 0$ as $\alpha\to+\infty$, we get 
 that $\bar{w}_1=0$ a.e. in $B_0\left(\frac{R}{2} \right)$ and then, we get \eqref{eq5}  using \eqref{eq8}. This ends the proof of Lemma \ref{controlv2}.
 \end{proof}
 
 As shown in Druet and Hebey \cite{kgmp3} and Hebey and Truong \cite{kgmp4}, phase compensation holds true when $n = 3, 4$. Let $\Phi$ be as in \eqref{eqphi}, and 
 $(B_\alpha)_\alpha$ be as in \eqref{bulle}. Since $0 \le \Phi \le \frac{1}{q}$, there holds that
 \begin{equation}\label{PhaseCompensatEqtPrel}
 0 \le \frac{\int_M\Phi(B_\alpha)B_\alpha^2dv_g}{\int_MB_\alpha^2dv_g} \le \frac{1}{q}
 \end{equation}
 for all $\alpha$. Phase compensation when $n = 3, 4$ expresses the fact that in these dimensions, the ratio in \eqref{PhaseCompensatEqtPrel} converges to zero as 
 $\alpha \to +\infty$. We prove in what follows that the limit of the ratio in \eqref{PhaseCompensatEqtPrel} 
 jumps from $0$ to $\frac{1}{q}$ when $n \ge 5$ and thus that, in this sense, we lose phase compensation.
 
 \begin{cor}[Loss of phase compensation]\label{LossPhaseComp} Let $(M,g)$ be a closed Riemannian $n$-manifold, $n \ge 5$. Then, 
 contrary to the $3$ and $4$-dimensional cases, 
 \begin{equation}\label{LossPhaseCompEqt}
\lim_{\alpha\to+\infty}\frac{\int_M\Phi(B_\alpha)B_\alpha^2dv_g}{\int_MB_\alpha^2dv_g} = \frac{1}{q}~,
 \end{equation}
where $(B_\alpha)_\alpha$ is as in \eqref{bulle}, 
and $\Phi$ is as in \eqref{eqphi}.
 \end{cor}
 
 \begin{proof}[Proof of Corollary \ref{LossPhaseComp}] By \eqref{RescBubbleCv} and Lemma \ref{controlv2}, if we 
 let $v_\alpha = \Phi(B_\alpha)$, and define $\hat{v}_\alpha$ to be the function given by 
$\hat{v}_\alpha(x)=v_\alpha(\exp_{x_\alpha}(\mu_\alpha x))$ for $x \in \mathbb{R}^n$, then $\hat v_\alpha \to \frac{1}{q}$ in $L^p_{loc}(\mathbb{R}^n)$ and 
a.e. as $\alpha \to +\infty$. In particular, we can write by the dominated convergence theorem that for $\delta > 0$ sufficiently small,
  \begin{equation}\label{PhaseCompProofEqt1}
  \begin{split}
  \int_M v_\alpha B_\alpha^2 dv_g 
  &= \int_{B_{x_\alpha}(\delta)} v_\alpha B_\alpha^2 dv_g + O(\mu_\alpha^{n-2}),\\
  &= \mu_\alpha^2 \int_{B_0(\delta/\mu_\alpha)} \hat{v}_\alpha B^2 dv_{\hat{g}_\alpha}+o(\mu_\alpha^2),\\
  &= \frac{\mu_\alpha^2}{q} \int_{B_0(\delta/\mu_\alpha)} B^2 dv_{\hat{g}_\alpha} +o(\mu_\alpha^2),
  \end{split}
  \end{equation}
 where $B$ is as in \eqref{RescBubbleCv}, and $\hat{g}_\alpha$ is as in \eqref{defhat}. Independently,
  \begin{equation}\label{PhaseCompProofEqt2}
  \begin{split}
  \int_M B_\alpha^2 dv_g 
&= \int_{B_{x_\alpha}(\delta)} B_\alpha^2 dv_g + O(\mu_\alpha^{n-2}),\\
&= \mu_\alpha^2\int_{B_0(\delta/\mu_\alpha)} B^2 dv_{\hat{g}_\alpha} +o(\mu_\alpha^2),
  \end{split}
  \end{equation}
 and we get \eqref{LossPhaseCompEqt} by combining \eqref{PhaseCompProofEqt1} and \eqref{PhaseCompProofEqt2}. This ends the 
 proof of Corollary \ref{LossPhaseComp}.
 \end{proof}
 
 We prove now that in the critical case of \eqref{kg}, when $n \ge 5$, we cannot hope for a $C^1$, and even a $C^0$, convergence of the potentials in the leading equation of 
 \eqref{kg} when dealing with blowing-up sequences of solutions $(u_\alpha,v_\alpha)$ of such systems. 
 More precisely, we let $u_\alpha$ and $v_\alpha$ be smooth positive functions such that 
 \begin{equation}\label{kgperturb}
\begin{cases}
\Delta_g u_\alpha+m_0^2 u_\alpha=u_\alpha^{2^\star-1}+\omega_\alpha^2(1-qv_\alpha)^2 u_\alpha\\
\Delta_g v_\alpha +(m_1^2+q^2u_\alpha^2) v_\alpha = q u_\alpha^2
\end{cases}
\end{equation}
 for all $\alpha$, where $(\omega_\alpha)_\alpha$ is a converging sequence in $(-m_0,+m_0)$. We assume that $(u_\alpha)_\alpha$ is bounded in 
 $H^1$. When $n = 3$, elliptic theory gives that $(v_\alpha)_\alpha$ is bounded in $C^{0,\theta}$ for some $\theta \in (0,1)$, and thus, up to a 
 subsequence, the $v_\alpha$'s converge in $C^0$ as $\alpha \to +\infty$. We claim that the convergence stops to hold true when $n \ge 5$.
 
\begin{cor}[Non $C^0$-convergence of the $v_\alpha$'s]\label{NonC0Conv} Let $(M,g)$ be a closed Riemannian $n$-manifold, $n \ge 5$, 
$(\omega_\alpha)_\alpha$ be a converging sequence in $(-m_0,+m_0)$, and $u_\alpha, v_\alpha > 0$ be smooth positive functions satsifying 
\eqref{kgperturb} for all $\alpha$. Assume that $(u_\alpha)_\alpha$ is bounded in $H^1$, and that $\Vert u_\alpha\Vert_{L^\infty} \to +\infty$ as $\alpha \to +\infty$. 
Then there are no subsequences of $(v_\alpha)_\alpha$ which 
converge in $C^0$.
\end{cor}

 \begin{proof}[Proof of Corollary \ref{NonC0Conv}] We assume by contradiction that, up to a passing to a subsequence, 
 $v_\alpha \to v$ in $C^0$. The sequence $(u_\alpha)_\alpha$ is bounded in $H^1$ 
 and, by \eqref{kgperturb}, it satisfies an equation like
 $$\Delta_gu_\alpha + h_\alpha u_\alpha = u_\alpha^{2^\star-1}~,$$
 where $(h_\alpha)_\alpha$ converges in $L^\infty$. In particular, the $H^1$-theory of Struwe \cite{Struwe} applies, see also Hebey \cite{HebeyBookZurich} for an exposition in book form 
 in this particular context, and we get that
 \begin{equation}\label{DecompH11stLem}
 u_\alpha = u_\infty + \sum_{i=1}^kB_\alpha^i + R_\alpha
 \end{equation}
 for all $\alpha$, where $u_\infty \in H^1$ is a weak solution of $\Delta_gu + h u = u^{2^\star-1}$, 
$h = \lim h_\alpha$, $k \in \mathbb{N}$, the $(B_\alpha^i)_\alpha$'s are bubbles as in \eqref{bulle}, and $R_\alpha \to 0$ in $H^1$ 
as $\alpha \to +\infty$. By the Tr\"udinger \cite{Trudinger} regularity theory, $u_\infty \in H_2^p$ for all $p \ge 1$. By the second equation in 
\eqref{kgperturb}, the sequence $(v_\alpha)_\alpha$ is bounded in $H^1$. Up to passing to a subsequence, we may 
assume that $u_\alpha \to u_\infty$ in $L^{2+\varepsilon}$ for $\varepsilon > 0$ sufficiently small, and that $v_\alpha \rightharpoonup v$ in $H^1$. Then we get that 
$v$ solves
$$\Delta_gv + (m_1^2+q^2u_\infty^2) v = q u_\infty^2$$
and by regularity theory, $v$ is $C^3$. If $x_0$ is a point where $v$ is maximum, $\Delta_gv(x_0) \ge 0$, and we thus get that 
\begin{equation}\label{EstimvLim1stLem}
v \le \frac{qu_\infty(x_0)^2}{m_1^2+q^2u_\infty(x_0)^2}~.
\end{equation}
The assumption $\|u_\alpha\|_{L^\infty}\to +\infty$ as $\alpha\to +\infty$ and an adaptation of the Tr\"udinger argument \cite{Trudinger} imply that $k\in\mathbb{N}^\star$ in \eqref{DecompH11stLem}. Let $\mu_\alpha = \min_i\mu_{i,\alpha}$, where the $\mu_{i,\alpha}$'s are the weights of the bubbles in \eqref{DecompH11stLem}. Up to renumbering, 
and up to passing to a subsequence, we can assume that $\mu_\alpha = \mu_{1,\alpha}$ for all $\alpha$. We let the $x_\alpha$'s be the centers of the bubble $(B_\alpha^1)_\alpha$. 
By rescaling arguments, proceeding as in Proposition 7.1 in Hebey \cite{HebeyBookZurich}, 
\begin{equation} \label{eql612Bis1stLemma}
 \mu_\alpha^{\frac{n-2}{2}} u_\alpha(\exp_{x_\alpha}(\mu_\alpha x)) \to B(x)
\end{equation} 
in $C^1_{loc}(\mathbb{R}^n)$ as $\alpha\to +\infty$, where $B$ is as in \eqref{RescBubbleCv}.
By \eqref{eql612Bis1stLemma} we can apply Lemma \ref{controlv2}. In particular, we get that if $x_1$ is the limit of the $x_\alpha$'s, then 
$v(x_1) = \frac{1}{q}$, a contradiction with \eqref{EstimvLim1stLem} which implies that $v < \frac{1}{q}$ everywhere in $M$. This ends the proof of 
Corollary \ref{NonC0Conv}. 
\end{proof}

  A typical example where Corollary \ref{NonC0Conv} applies is when $(M,g) = (S^n,g)$ is the unit $n$-sphere, $\omega_\alpha = 0$ for all $\alpha$, 
  $v_\alpha = \Phi(u_\alpha)$, where $\Phi$ is as in \eqref{eqphi}, $m_0^2 = \frac{n(n-2)}{4}$, and 
   \begin{equation}\label{bullesphere}
 u_\alpha(x)= \left(\frac{n(n-2)}{4}(\beta_\alpha^2-1) \right)^{\frac{n-2}{4}}(\beta_\alpha-\cos(d_g(x_0,x)))^{-\frac{n-2}{2}}
 \end{equation}
for all $\alpha$, some $x_0 \in S^n$, and $\beta_\alpha$'s such that $\beta_\alpha > 1$ for all $\alpha$, and $\beta_\alpha \to 1$ as $\alpha \to +\infty$. 
 
 \section{Differentiablity of the auxiliary energy}\label{DiffAuxilEnergy}
 
 We return in this section to the map $\Phi$ we constructed in {\eqref{eqphi} and prove that despite the fact that $\Phi$ is not a priori 
 differentiable, the map $\Psi: H^1 \to \mathbb{R}$ given by
 \begin{equation}\label{defpsi}
 \Psi(u) = \frac{1}{2}\int_M\left(1-q\Phi(u)\right)u^2dv_g
 \end{equation} 
 is $C^1$ with a nice differential given by the jumping of the square power on $u$ to a square power on $1-q\Phi(u)$. More 
 precisely, we prove that the following lemma holds true.
 
\begin{lem}\label{l2} Let $(M,g)$ be a closed Riemannian $n$-manifold, $n\geq 5$. Let also $q$, $m_1>0$, 
$\Phi : H^1\to H^1$ be as in \eqref{eqphi}, and $\Psi : H^1 \to \mathbb{R}$ be defined by \eqref{defpsi}. Then $\Psi$ is $C^1$ in $H^1$ and 
\begin{equation} \label{eql2}
D\Psi(u).(\varphi)=\int_M (1-q\Phi(u))^2 u\varphi dv_g
\end{equation} 
for all $u, \varphi\in H^1$.
\end{lem}

\begin{proof}[Proof of Lemma \ref{l2}] It suffices to discuss the differentiability of
$$\Psi_R(u)=\int_M u^2 \Phi(u) dv_g~.$$
We compute
\begin{equation}\label{1l2}
\Psi_R(u+\varphi)= \int_Mu^2 \Phi(u+\varphi) dv_g +2 \int_M u\varphi \Phi(u) dv_g +o(\|\varphi\|_{H^1})
\end{equation}
since, using H\"older inequality, and the inequality $0\leq\Phi\leq 1/q$, there holds that
\begin{equation}\label{2l1}
\begin{split}
&\int_M |u\varphi (\Phi(u+\varphi)-\Phi(u))| dv_g\\
&\leq \|u\|_{L^{2^\star}} \|\varphi\|_{L^{2^\star}}\|\Phi(u+\varphi)-\Phi(u)\|_{L^{\frac{2^\star}{2^\star-2}}}\\
&\leq C \|u\|_{L^{2^\star}} \|\varphi\|_{L^{2^\star}}\|\Phi(u+\varphi)-\Phi(u)\|_{L^{2^\star}} \text{ if $5\leq n \leq 6$,}\\
&\leq C \|u\|_{L^{2^\star}} \|\varphi\|_{L^{2^\star}}\|\Phi(u+\varphi)-\Phi(u)\|_{L^{2^\star}}^{2^\star-2} \text{ if $n\geq 7$.}
\end{split}
\end{equation}
Testing the equations satisfied by $\Phi(u)$ and $\Phi(u+\varphi)$ in the restricted weak sense against $\Phi(u+\varphi)$ and $\Phi(u)\in L^\infty\cap H^1$, we can write that
\begin{equation*}
\begin{split}
& \int_M \langle \nabla\Phi(u),\nabla \Phi(u+\varphi)\rangle dv_g +m_1^2 \int_M \Phi(u) \Phi(u+\varphi) dv_g\\
& +   q^2 \int_M u^2 \Phi(u) \Phi(u+\varphi) dv_g  =q\int_M u^2 \Phi(u+\varphi) dv_g
\end{split}
\end{equation*}
and that
\begin{equation*}
\begin{split}
& \int_M \langle \nabla\Phi(u),\nabla \Phi(u+\varphi)\rangle dv_g +m_1^2 \int_M \Phi(u) \Phi(u+\varphi)dv_g\\
&+ q^2\int_M \Phi(u+\varphi) \Phi(u)(u^2+2u\varphi) dv_g  =q\int_M \Phi(u) (u^2+2u\varphi) dv_g +o(\|\varphi\|_{H^1})~.
\end{split}
\end{equation*}
We eliminate the gradient terms in these two equations and get that
\begin{equation*}
\begin{split}
\int_{M} u^2 \Phi(u+\varphi) dv_g=&\Psi_R(u)+ 2\int_M \varphi u \Phi(u) \left(1-q\Phi(u) \right) dv_g\\
& +2q\int_M \Phi(u) u \varphi (\Phi(u)-\Phi(u+\varphi))dv_g+o(\|\varphi\|_{H^1}).
\end{split}
\end{equation*}
Proceeding as in \eqref{2l1}, it follows from the Sobolev inequality that
$$ \int_M u^2 \Phi(u+\varphi) dv_g=\Psi_R(u) +2 \int_M u \varphi \Phi(u) (1-q\Phi(u)) dv_g +o(\|\varphi\|_{H^1}).$$
Using now \eqref{1l2}, we get
$$\Psi(u+\varphi)=\Psi(u) +\int_M (1-q\Phi(u))^2 u\varphi dv_g +o(\|\varphi\|_{H^1})$$
for $u,\varphi\in H^1$, and \eqref{eql2} holds true. The continuity of $D\Psi$ easily follows from the continuity of $\Phi$. Lemma \ref{l2} is proved.
\end{proof}

\section{Existence of mountain pass solutions}\label{MountainPassSol}

Formally, solutions of \eqref{kg} are critical points of the functional $S$ defined by
\begin{equation*}
\begin{split}
S(u,v)=&\frac{1}{2}\int_{M}|\nabla u|^2 dv_g -\frac{\omega^2}{2}\int_{M}|\nabla v|^2 dv_g +\frac{m_0^2}{2}\int_M u^2 dv_g\\
& -\frac{\omega^2 m_1^2}{2}\int_M v^2 dv_g -\frac{1}{2^\star} \int_M u^p dv_g-\frac{\omega^2}{2}\int_{M} u^2(1-qv)^2 dv_g.
\end{split}
\end{equation*} 
 We face here two major difficulties : the functional $S$ is strongly indefinite (because of the competition between $u$ and $v$) and it does not make sense for all $u,v\in H^1$ when $n\geq 5$ (since then $2^\star<4$). For instance, the expression $\int_M u^2 v^2 dv_g$ of the last term of $S$ does not make sense for all $u,v\in H^1$. We let $\Phi$ be defined as in \eqref{eqphi} and introduce the functional $I: H^1 \to \mathbb{R}$ given by
 \begin{equation}\label{defI}
\begin{split}
I(u)=&\frac{1}{2}\int_M |\nabla u|^2 dv_g +\frac{m_0^2}{2}\int_M u^2 dv_g\\
      &-\frac{1}{p}\int_M (u^+)^{p} dv_g  -\frac{\omega^2}{2}\int_M (1-q\Phi(u))u^2 dv_g~,
\end{split}
\end{equation} 
where $u^+ = \max(u,0)$, and $p \in (2,2^\star]$. The functional makes sense in any dimension since $0 \le \Phi \le \frac{1}{q}$, it is $C^1$ by Lemma \ref{l2}, and still by Lemma 
\ref{l2}, if $u$ is a nonnegative critical point of $I$, then 
$\left(u,\Phi(u)\right)$ solves \eqref{kg}. We define a mountain pass solution in Definition \ref{mp} below. The mountain-pass lemma we refer to 
in this definition is the one given in Ambrosetti and Rabinowitz  \cite{AmbrosettiRab}. 

 \begin{defi}\label{mp} A couple  $(u,v)$ is a \textbf{mountain-pass solution} of \eqref{kg} if $u\in H^1$, $v=\Phi(u)$, with 
 $\Phi$ given in Lemma \ref{l1}, and $u$ is obtained from $I$, defined in \eqref{defI}, by the mountain-pass lemma from $0$ to $u_1\in H^1$, where $I(u_1)<0$.
  \end{defi}
  
  When $p \in (2,2^\star)$ the existence part in Theorem \ref{SubCritThm} very easily follows from the mountain pass lemma of Ambrosetti and Rabinowitz \cite{AmbrosettiRab} and 
  the compactness of the embedding $H^1 \subset L^p$. We very briefly discuss the proof in what follows.
  
  \begin{proof}[Proof of the existence part in Theorem \ref{SubCritThm}] By Lemma \ref{l2}, the function $I$ defined by \eqref{defI} is $C^1$ in $H^1$ . 
Obviously, since $p > 2$, and since $0 \le \Phi \le \frac{1}{q}$ and $\omega^2 < m_0^2$, there exist $\rho_1, \rho_2 > 0$, $\rho_1 \ll 1$, such that $I(u) \ge \rho_2$ for all 
$u \in H^1$ satisfying that $\Vert u\Vert_{H^1} = \rho_1$. Let $\bar{u}_0 \in H^1$, $\bar{u}_0^+ \not\equiv 0$ and $T_0 \gg 1$ be such that $I(T_0\bar{u}_0)<0$. Since 
$I(0)=0$ and $I(T_0 \bar{u}_0)<0$, we can apply the mountain-pass lemma of Ambrosetti and Rabinowitz \cite{AmbrosettiRab}, and we 
get that there exists a Palais-Smale sequence $(u_\alpha)_\alpha$ at level
$$c=\underset{\gamma\in \Gamma }{\inf}\underset{u\in \gamma}{\sup} I(u)~,$$
where $\Gamma$ stands for the set of continuous paths from $0$ to $T_0 \bar{u}_0$. Writing that $DI(u_\alpha).(u_\alpha^-) = o\left(\Vert u_\alpha^-\Vert_{H^1}\right)$, we get 
that $u_\alpha^- \to 0$ in $H^1$ as $\alpha \to +\infty$. Following the classical scheme in Br\'ezis and Nirenberg \cite{BrezisNiremberg}, 
it follows that the sequence $(u_\alpha)_\alpha$ 
is bounded in $H^1$. By the Rellich-Kondrakov theorem, passing to a subsequence, we get that there exists $u \in H^1$ such that $u_\alpha \rightharpoonup u$ 
in $H^1$, $u_\alpha \to u$ in $L^p$, and $u_\alpha \to u$ a.e. By \eqref{EstPhiHolderType}, $\Phi(u_\alpha) \to \Phi(u)$ in $H^1$, and we conclude with very 
standard arguments that $\left(u,\Phi(u)\right)$ is the mountain pass solution we look for. This ends the proof of the existence part in Theorem \ref{SubCritThm}.  
  \end{proof}
 
 In the critical case where $p = 2^\star$, the above proof needs to be refined. 
 The following result follows from applying the mountain pass lemma of Ambrosetti and Rabinowitz \cite{AmbrosettiRab} together with 
  arguments from Aubin \cite{Aubin} and Br\'ezis and Nirenberg \cite{BrezisNiremberg}. We let $K_n$ be the sharp constant 
  for the 
 standard Euclidean Sobolev inequality 
 $\|u\|_{L^{2^\star}}\leq K_n \|\nabla u\|_{L^2}$ with $u\in H^1(\mathbb{R}^n)$. The explicit value of $K_n$ is known and given by $n(n-2)\omega_n^{2/n}K_n^2=4$, where $\omega_n$ is the volume of the unit $n$-sphere $S^n$ endowed with its canonical metric. 
  A compact version of this sharp inequality is in Hebey and Vaugon \cite{HebeyVaugon}.
 
 \begin{lem}\label{brezis}
 Let $\bar{u}_0 \in H^1$, $\bar{u}_0^+ \not\equiv 0$, and $T_0 \gg 1$ be such that $I(T_0\bar{u}_0)<0$. Let $p = 2^\star$ and
  \begin{equation}\label{chemin}
 c=\underset{\gamma\in \Gamma }{\inf}\underset{u\in \gamma}{\sup} I(u)~,
\end{equation} 
where $\Gamma$ stands for the set of continuous paths from $0$ to $T_0 \bar{u}_0$.
Assume that 
\begin{equation}\label{eqbrezis}
c<\frac{1}{nK_n^n}~.
\end{equation}
Then there exists a smooth mountain-pass solution $(u,v)$ of \eqref{kg}, with $u$, $v>0$ in $M$.
 \end{lem}
 
 \begin{proof}[Proof of Lemma \ref{brezis}] First, we apply the mountain-pass lemma to get $u$. By Lemma \ref{l2}, the function $I$ defined by \eqref{defI} is $C^1$ in $H^1$ . 
Obviously, since $p > 2$, and since $0 \le \Phi \le \frac{1}{q}$ and $\omega^2 < m_0^2$, there exist $\rho_1, \rho_2 > 0$, $\rho_1 \ll 1$, such that $I(u) \ge \rho_2$ for all 
$u \in H^1$ satisfying that $\Vert u\Vert_{H^1} = \rho_1$. Since 
$I(0)=0$ and $I(T_0 \bar{u}_0)<0$, we can apply the mountain-pass lemma of Ambrosetti and Rabinowitz \cite{AmbrosettiRab}, and we 
get that there exists a sequence $(u_\alpha)_\alpha$ of functions in $H^1$ such that
\begin{eqnarray}
&&I(u_\alpha) \to  c,\label{eq22}\\
&&DI(u_\alpha) \to  0 \text{ in } (H^1)',\label{eq23}\\
&&DI(u_\alpha).(u_\alpha)  =  o(\|u_\alpha\|_{H^1}),\label{eq24}
\end{eqnarray}
 as $\alpha \to +\infty$, where $c$ is as in \eqref{chemin}. Applying \eqref{eq23} to the $(u_\alpha^-)$'s and using $\omega^2<m_0^2$, we get that
 \begin{equation}\label{eq29}
 u_\alpha^- \to 0 \text{ in } H^1
 \end{equation}
 as $\alpha \to+\infty$. 
Using \eqref{eq29}, $0\leq \Phi(u_\alpha) \leq 1/q$, the Sobolev and H\"older inequalities, we get combining  \eqref{eq22} and \eqref{eq24} that
\begin{equation*}
\left(\frac{1}{2}-\frac{1}{2^\star} \right) \int_M |u_\alpha|^{2^\star}dv_g=c+o(1)+o(\|u_\alpha\|_{H^1})+O\left(\|u_\alpha\|_{L^{2^\star}}^2\right).
\end{equation*}
As a consequence, we get $\|u_\alpha\|_{L^{2^\star}}\leq C+o(\|u_\alpha\|_{H^1})$ for $C>0$ independent of $\alpha$, and then
\begin{equation}\label{eq33}
\|u_\alpha\|_{H^1}=O(1)
\end{equation}
by using \eqref{eq24}, \eqref{eq29} and $m_0^2>\omega^2$ again.
 Up to a subsequence, there exists $u\in H^1$ such that 
  \begin{equation}\label{eq25}
 \begin{array}{l} 
 u_\alpha \rightharpoonup u \text{ weakly in } H^1,\\
 u_\alpha \to u \text{ in }L^2,\\
 u_\alpha \to u\text{ and  }(u_\alpha^+)\to (u^+) \text{ a.e.}
 \end{array}
 \end{equation}
 Thus, by \eqref{eq29}, we have $u\geq 0$. By the Sobolev embedding theorem and \eqref{eq33}, 
 the sequence $((u_\alpha^+)^{2^\star-1})_\alpha$ is bounded in $L^{\frac{2^\star}{2^\star -1}}$ and by \eqref{eq25}, 
 we get $(u_\alpha^+)^{2^\star-1} \rightharpoonup u^{2^\star-1}$, weakly in $L^{\frac{2^\star}{2^{\star}-1}}$ as $\alpha \to +\infty$. Then, 
 using \eqref{EstPhiHolderType}, \eqref{eql2} and \eqref{eq25}, letting $\alpha$ go to $+\infty$ in \eqref{eq23}, we get that for any $\varphi\in H^1$
  \begin{equation}\label{eq26}
  \begin{split}
& \int_M \langle \nabla u, \nabla \varphi \rangle_{g} dv_g +m_0^2 \int_M u \varphi dv_g\\
 &= \int_M u^{2^\star-1}\varphi dv_g  +\omega^2 \int_M u(1-q \Phi(u))^2\varphi dv_g,
  \end{split}
 \end{equation}
 where $u$ is as in \eqref{eq25}. In other words, $u$ satisfies the first equation of \eqref{kg} in a weak sense. It remains to prove that $u\not\equiv 0$. 
We assume by contradiction that $u\equiv 0$. By \eqref{eq33}, up to a subsequence, we can assume that for some $t\geq 0$,
  \begin{equation}\label{eq30}
 \int_M |\nabla u_\alpha|^2 dv_g \to t
 \end{equation}
as $\alpha \to +\infty$. Then, using \eqref{eq24}, \eqref{eq29}, $0\leq \Phi(u_\alpha) \leq 1/q$ and \eqref{eq25}, we can write 
  \begin{equation}\label{eq31}
 \int_M u_\alpha^{2^\star} dv_g \to t
 \end{equation}
as $\alpha \to +\infty$. Using now  \eqref{eqbrezis} and \eqref{eq22}, we have
  \begin{equation}\label{eq32}
 0< \frac{t}{n}=c< \frac{1}{n K_n^n},
 \end{equation}
 where $t$ is as in \eqref{eq30} and \eqref{eq31}. On the other hand, keeping in mind that $u\equiv 0$ and writing the optimal Sobolev inequality 
 in Hebey and Vaugon \cite{HebeyVaugon} for the $u_\alpha$'s, we get
  \begin{equation}
 t^{\frac{2}{2^\star}}\leq K_n^2 t~.
 \end{equation}
In particular, we get a contradiction with \eqref{eq32}. This proves that $u \not\equiv 0$. 
General regularity results as in Gilbarg and Tr\"udinger \cite{Gilbarg}, 
the Tr\"udinger \cite{Trudinger} critical regularity result, and the maximum principle then apply. In particular, $u$ and $v=\Phi(u)$ are smooth, positive in $M$, and they satisfy \eqref{kg}. 
 This proves Lemma \ref{brezis}.
 \end{proof}

 We use now the test functions introduced by Aubin \cite{Aubin}. 
 Given $x_0\in M$, $\varepsilon>0$ and $\rho_0>0$, we define, for $x\in M$, the function $u_\varepsilon$ by
  \begin{equation}\label{defueps}
\begin{cases}
 u_\varepsilon(x)=\left(\frac{\varepsilon}{\varepsilon^2+r^2} \right)^{\frac{n-2}{2}}-\left(\frac{\varepsilon}{\varepsilon^2+\rho_0^2} \right)^{\frac{n-2}{2}} &\text{ if } r\leq \rho_0,\\
 u_\varepsilon(x) = 0 &\text{ if } r\geq \rho_0,
 \end{cases}
 \end{equation}
 where $r=d_g(x,x_0)$ is the geodesic distance induced in $M$ by the metric $g$. Then, computing as in Aubin \cite{Aubin}, for any $\lambda\in \mathbb{R}$ :
  \begin{equation}\label{Aubin}
 I_\lambda(u_\varepsilon)=\frac{1}{K_n^2}\left(1-C_1\left(\frac{n-2}{4(n-1)}S_g(x_0)-\lambda \right)\varepsilon^2+o(\varepsilon^2) \right)\text{ for }n\geq 5,\\
 \end{equation}
 as $\varepsilon\to 0$, where 
  \begin{equation}\label{Aubin2}
 I_\lambda(u)=\frac{\int_M (|\nabla u|^2+\lambda u^2) dv_g}{(\int_M |u|^{2^\star})^{2/2^\star}}
 \end{equation}
 for $u\in H^1\backslash \{0\}$, and $C_1, C_2>0$ are independent of $\varepsilon$. There also holds
  \begin{equation}\label{Aubin3}
 \begin{split}
 \int_M u_\varepsilon^{2^\star} dv_g=\int_{\mathbb{R}^n}\left(\frac{1}{1+|x|^2}\right)^n dx +o(1),\\
 \int_M |\nabla u_\varepsilon|^2 dv_g=n(n-2) \int_Mu_\varepsilon^{2^\star}dv_g + o(1)
 \end{split}
 \end{equation}
 as $\varepsilon\to 0$. We are now in position to prove the existence part in Theorem \ref{CritThm}.

\begin{proof}[Proof of the existence part in Theorem \ref{CritThm}] By Druet and Hebey \cite{kgmp3} and Hebey and Truong \cite{kgmp4}, we only need to address the 
case $n \ge 5$. By \eqref{prel1}, we can choose $x_0\in M$ such that 
  \begin{equation}\label{hypexist}
 m_0^2<\frac{n-2}{4(n-1)}S_g(x_0).
 \end{equation}
 Let $(\varepsilon_\alpha)_\alpha$ be a sequence of positive real numbers such that $\varepsilon_\alpha \to 0$ 
 as $\alpha \to +\infty$,  $u_\alpha=u_{\varepsilon_\alpha}$, where $u_{\varepsilon_\alpha}$ is given in \eqref{defueps}, and $\mathcal{H}$ be the functional defined in $H^1$ by
  \begin{equation}\label{eq17}
 \mathcal{H}(u)= \frac{1}{2}\int_M |\nabla u|^2 dv_g +\frac{m_0^2}{2} \int_M u^2 dv_g -\frac{1}{2^\star} \int_M |u|^{2^\star} dv_g.
 \end{equation}
 By \eqref{Aubin3}, there exists $T_0\gg 1$ such that $I(T_0 u_\alpha)<0$ for all $\alpha\gg 1$, where $I$ is as in \eqref{defI}. Using \eqref{Aubin}, we can choose 
 $\alpha$ sufficiently large to have, setting $\bar{u}_0=u_\alpha$,
 \begin{equation}\label{eq19}
 I_{m_0^2}(\bar{u}_0)<\frac{1}{K_n^2},
 \end{equation}
  where $I_{m_0^2}$ is given in \eqref{Aubin2}. There holds that
  \begin{equation}\label{eq20}
 \max_{0\leq t \leq T_0} I(t \bar{u}_0) \leq  \max_{0\leq t \leq T_0}\mathcal{H}(t \bar{u}_0)\leq \frac{1}{n} I_{m_0^2}(\bar{u}_0)^{\frac{n}{2}} < \frac{1}{nK_n^n},
 \end{equation}
 where $\mathcal{H}$ is defined in \eqref{eq17}. We get the first inequality since $\bar{u}_0$ is non-negative and $0\leq \Phi(\bar{u}_0) \leq \frac{1}{q}$ (Lemma \ref{l1}), the second one comes from maximizing the function $t\mapsto \mathcal{H}(t \bar{u}_0)$ on $\mathbb{R}_+$ and the last one is given by \eqref{eq19}. In particular, Lemma \ref{brezis} applies, 
 and this ends the proof of the existence part in 
 Theorem \ref{CritThm}.
 \end{proof}
 
 Existence of solutions and semiclassical limits for systems like \eqref{kg}, in 
Euclidean space, have  
been investigated by 
D'Aprile and Mugnai \cite{AprMug1,AprMug2}, 
D'Aprile and Wei \cite{AprWei1,AprWei2}, 
D'Avenia and Pisani \cite{AvePis}, 
D'Avenia, Pisani and Siciliano \cite{AvePisSic,AvePisSic2}, 
Azzollini, Pisani and Pomponio \cite{AzzPisPom}, 
Azzollini and Pomponio \cite{AzzPom2}, 
Benci and Bonanno \cite{BenBon}, 
Benci and Fortunato \cite{BenFor0,BenFor01,BenFor2,BenFor4}, 
Cassani \cite{Cas}, 
Georgiev and Visciglia \cite{GeoVis} and Mugnai \cite{Mug,Mug2}. 
 
 \section{A Priori bounds in the subcritical case}\label{AprioriSubCrit}
 
 Let $(\omega_\alpha)$ be a sequence in $(-m_0,m_0)$ such that $\omega_\alpha\to \omega$ as $\alpha\to +\infty$ for some $\omega\in [-m_0,m_0]$. Also let $p\in (2,2^\star)$ and $((u_\alpha,v_\alpha))_\alpha$ be a sequence of smooth positive solutions of \eqref{kg} with phase $\omega_\alpha$. Then, 
\begin{equation}
\begin{cases}\label{Sub1}
\Delta_g u_\alpha+m_0^2 u_\alpha=u_\alpha^{p-1}+\omega_\alpha^2(1-qv_\alpha)^2 u_\alpha,\\
\Delta_g v_\alpha +(m_1^2+q^2u_\alpha^2)v_\alpha=qu_\alpha^2,
\end{cases}
\end{equation}
for all $\alpha$. By the Lemma \ref{l1}, $0\leq v_\alpha\leq \frac{1}{q}$ for all $\alpha$. Assume by contradiction that
\begin{equation}\label{Sub2}
\max_M u_\alpha \to +\infty
\end{equation}
as $\alpha\to +\infty$. Let $x_\alpha\in M$ and $\mu_\alpha>0$ given by
$$u_\alpha(x_\alpha)=\max_M u_\alpha =\mu_\alpha^{-\frac{2}{(p-2)}}. $$
By \eqref{Sub2}, $\mu_\alpha\to 0$ as $\alpha\to +\infty$. Define $\tilde{u}_\alpha$ by
$$\tilde{u}_\alpha(x)=\mu_\alpha^{\frac{2}{p-2}} u_\alpha\left(\exp_{x_\alpha}(\mu_\alpha x)\right) $$
and $g_\alpha$ by $g_\alpha(x)=\left(\exp_{x_\alpha}^\star g \right)(\mu_\alpha x)$ for $x\in B_0(\delta\mu_\alpha^{-1})$, where $\delta>0$ is small. Since $\mu_\alpha\to 0$, we get that $g_\alpha\to \xi$ in $C^2_{loc}(\mathbb{R}^n)$ as $\alpha\to +\infty$. Moreover, by \eqref{Sub1}
\begin{equation}\label{Sub3}
\Delta_{g_\alpha} \tilde{u}_\alpha+\mu_\alpha^2 m_0^2 \tilde{u}_\alpha=\tilde{u}_\alpha^{p-1}+\mu_\alpha^2\omega_\alpha^2(1-q\hat{v}_\alpha)^2 \tilde{u}_\alpha~~,
\end{equation}
where $\hat{v}_\alpha$ is given by $\hat{v}_\alpha(x)=v_\alpha\left(\exp_{x_\alpha} (\mu_\alpha x)\right)$. We have $\tilde{u}_\alpha(0)=1$ and $0\leq \tilde{u}_\alpha\leq 1$. By \eqref{Sub3} and standard elliptic theory arguments, we can write that, after passing to a subsequence, $\tilde{u}_\alpha\to u$ in $C^{1,\theta}_{loc}(\mathbb{R}^n)$ as $\alpha\to +\infty$ for some $\theta\in (0,1)$, where $u$ is such that $u(0)=1$ and $0\leq u\leq 1$. Then, 
$$\Delta u=u^{p-1} $$
in $\mathbb{R}^n$, where $\Delta$ is the Euclidean Laplacian. It follows that $u$ is actually smooth and positive, and, since $2<p<2^\star$, we get a contradiction with the Liouville result of Gidas and Spruck \cite{GidasSpruck}. As a conclusion, \eqref{Sub2} is not possible and there exists $C>0$ such that
\begin{equation}\label{Sub4}
u_\alpha+v_\alpha\leq C
\end{equation}
in $M$ for all $\alpha$. Coming back to \eqref{Sub1}, it follows that the sequences $(u_\alpha)_\alpha$ and $(v_\alpha)_\alpha$ are actually bounded in $H^{2,s}$ for all $s$. Classical subcritical bootstrap argument and the Sobolev embeddings give that they are also bounded in $C^{2,\tilde{\theta}}$, $0<\tilde{\theta}<1$. This ends the proof of the uniform bounds in Theorem \ref{SubCritThm}.

 \section{Blow-up theory in the critical case - Bounded potentials.}\label{BlUpThryPtCtr}
 
  In what follows, we let $(M,g)$ be a closed Riemannian $n$-manifold, $n\geq 5$, $m_0,m_1, q>0$, and $(\omega_\alpha)_\alpha$ be a 
 sequence in $(-m_0,m_0)$ such that $\omega_\alpha\to \omega$ as $\alpha\to +\infty$ for some $\omega\in [-m_0,m_0]$. Also, we let 
 $\left((u_\alpha,v_\alpha)\right)_\alpha$ be a sequence of smooth positive solutions of \eqref{kg} in the critical case $p=2^\star$ with phases $\omega_\alpha$. Namely
\begin{equation}\label{kgstab}
\begin{cases}
&\Delta_g u_\alpha+m_0^2 u_\alpha=u_\alpha^{2^\star-1}+\omega_\alpha^2(1-qv_\alpha)^2 u_\alpha,\\
&\Delta_g v_\alpha+(m_1^2+q^2u_\alpha^2) v_\alpha= q u_\alpha^2,
\end{cases}
\end{equation} 
for all $\alpha$. Recall that we have a uniform bound in $L^\infty$ for the $v_\alpha$'s, to be more precise $0\leq v_\alpha \leq \frac{1}{q}$ for all $\alpha$. In 
particular, if we let
\begin{equation}\label{h}
h_\alpha=m_0^2-\omega_\alpha^2(qv_\alpha-1)^2~,
\end{equation}
then $\|h_\alpha\|_{L^\infty}\leq C$ for all $\alpha$, where $C>0$ is independent of $\alpha$. We assume here that
\begin{equation}\label{maxu}
\max_M u_\alpha \to +\infty
\end{equation}
as $\alpha\to +\infty$. Then, a priori, see Corollary \ref{NonC0Conv}, the $h_\alpha$'s do not converge in $L^\infty$. We let $(x_\alpha)_\alpha$ 
be a sequence of points in $M$
and $(\rho_\alpha)_\alpha$ be a sequence of positive real numbers, with $0<\rho_\alpha<i_g/7$, where $i_g$ is the injectivity radius of $(M,g)$. We assume 
that the $x_\alpha$'s and $\rho_\alpha$'s satisfy
 \begin{equation}\label{hyp}
\begin{cases}
&\nabla u_\alpha(x_\alpha)=0 ~~\text{for all } \alpha,\\
&d_g(x_\alpha,x)^{\frac{n-2}{2}}u_\alpha(x) \leq C~~ \text{for all } x\in B_{x_\alpha}(7\rho_\alpha) \text{ and all } \alpha,\\
&\lim_{\alpha\to +\infty} \rho_\alpha^{\frac{n-2}{2}}\underset{B_{x_\alpha}(6\rho_\alpha)}{\sup}u_\alpha=+\infty.
\end{cases}
\end{equation}
We let $\mu_\alpha$ be given by
\begin{equation}\label{mu}
\mu_\alpha=u_\alpha(x_\alpha)^{-\frac{2}{n-2}}
\end{equation}
and let $u_0$ be given by
\begin{equation}\label{u0}
u_0(x)= \left(\frac{1}{1+\frac{|x|^2}{n(n-2)}}\right)^{\frac{n-2}{2}}
\end{equation} 
for all $x \in \mathbb{R}^n$. 
The $u_\alpha$'s satisfy the stationary Schr\"odinger equation
\begin{equation}\label{SchroSingleEqt}
\Delta_gu_\alpha + h_\alpha u_\alpha = u_\alpha^{2^\star-1}
\end{equation}
for all $\alpha$, where $h_\alpha$ is as in \eqref{h}, and by Lemma \ref{l1}, as already mentioned, there exists $C > 0$ such that 
$\Vert h_\alpha\Vert_{L^\infty} \le C$ for all $\alpha$. The $L^\infty$-bound on the potentials in \eqref{SchroSingleEqt} makes 
that we can apply the $C^0$-estimates proved in Chapter 6 of the book of Hebey \cite{HebeyBookZurich} (see Druet \cite{Dcompactness}, Druet, Hebey, Robert \cite{DHRBlowup}, Druet, Hebey and V\'etois \cite{JFA}, Li-Zhu \cite{lizhu}, and Li-Zhang \cite{LiZ,LiZ2,LiZYamabe2} for original references). In particular, 
Lemma \ref{l61} and Proposition \ref{p61} below hold true. First we state Lemma \ref{l61}. 

 \begin{lem}[See Hebey \cite{HebeyBookZurich}]\label{l61}
 Let $(M,g)$ be a closed Riemannian $n$-manifold, $n\geq 5$, and $((u_\alpha,v_\alpha))_\alpha$ be a sequence of smooth positive solutions of \eqref{kgstab}, such that \eqref{maxu} holds true. Let $(x_\alpha)_\alpha$ and $(\rho_\alpha)_\alpha$ be such that \eqref{hyp} holds true. After passing to a subsequence
 \begin{equation}\label{rhomu}
\frac{\rho_\alpha}{\mu_\alpha}\to +\infty \quad \text{and} \quad \mu_\alpha\to 0
\end{equation}
as $\alpha\to +\infty$, where $\mu_\alpha$ is given by \eqref{mu}, and
\begin{equation} \label{eql61}
 \mu_\alpha^{\frac{n-2}{2}} u_\alpha\left(\exp_{x_\alpha}(\mu_\alpha .)\right) \to u_0,
\end{equation} 
in $C^1_{loc}(\mathbb{R}^n)$ as $\alpha\to +\infty$, where $u_0$ is given by \eqref{u0}.
\end{lem}

In order to state Proposition \ref{p61} we need to introduce the range $r_\alpha$ attached to the $u_\alpha$'s. 
We define $\varphi_\alpha : [0,\rho_\alpha) \to \mathbb{R}_+$ by
\begin{equation}\label{eq38}
\varphi_\alpha(r) =\frac{1}{|\partial B_{x_\alpha}(r)|} \int_{\partial B_{x_\alpha}(r)} u_\alpha d\sigma_g
\end{equation}
where $|\partial B_{x_\alpha}(r)|$ is the volume of the geodesic sphere of center $x_\alpha$. As a consequence of Lemma \ref{l61}, we have that
\begin{equation}\label{eq39}
(\mu_\alpha r)^{\frac{n-2}{2}} \varphi(\mu_\alpha r) \to \left(\frac{r}{1+\frac{r^2}{n(n-2)}} \right)^{\frac{n-2}{2}}
\end{equation}
in $C^1_{loc}(\mathbb{R}_+)$ as $\alpha\to +\infty$. Then we define $r_\alpha\in[2R_0 \mu_\alpha, \rho_\alpha]$ by 
\begin{equation}\label{eq40}
r_\alpha=\sup\left\{r\in [2R_0 \mu_\alpha, \rho_\alpha]\quad s.t. \left(s^{\frac{n-2}{2}} \varphi_\alpha(s) \right)'\leq 0  \text{ in }[2R_0 \mu_\alpha, r]\right\},
\end{equation}
where $R_0^2=n(n-2)$. Thanks to \eqref{eq39}, we have that 
\begin{equation}\label{eq41}
\frac{r_\alpha}{\mu_\alpha} \to +\infty
\end{equation}
as $\alpha\to +\infty$, while the definition of $r_\alpha$ gives that
\begin{equation}\label{eq42}
r^{\frac{n-2}{2}}\varphi_\alpha \text{  is non-increasing in } [2R_0\mu_\alpha,r_\alpha]
\end{equation}
and that
\begin{equation}\label{eq43}
\left(r^{\frac{n-2}{2}}\varphi_\alpha(r)\right)'(r_\alpha)=0 \text{ if }r_\alpha<\rho_\alpha.
\end{equation}
Proposition \ref{p61} 
gives sharp pointwise estimates on the $u_\alpha$'s at a distance like $r_\alpha$ of $x_\alpha$ (an infinitesimal version of the $C^0$-estimates in Druet, Hebey and 
Robert \cite{DHRBlowup}).

\begin{prop}[See Hebey \cite{HebeyBookZurich}]\label{p61}
 Let $(M,g)$ be a closed Riemannian $n$-manifold, $n\geq 5$, and $((u_\alpha,v_\alpha))_\alpha$ be a sequence of smooth positive solutions of \eqref{kgstab}, such that \eqref{maxu} holds true. Let $(x_\alpha)_\alpha$ and $(\rho_\alpha)_\alpha$ be such that \eqref{hyp} holds true. Let $R>0$ be such that $Rr_\alpha \leq 6 \rho_\alpha$ for all $\alpha\gg 1$. There exist $C>0$ and $(\varepsilon_\alpha)_\alpha$ a sequence of positive real numbers such that, after passing to a subsequence,
\begin{equation}\label{eqp61}
\begin{split}
&u_\alpha(x) +d_g(x_\alpha,x)|\nabla u_\alpha(x)|\leq C\mu_\alpha^{\frac{n-2}{2}} d_g(x_\alpha,x)^{2-n}, \\
&|u_\alpha(x)-B_\alpha(x)|\leq C\mu_\alpha^{\frac{n-2}{2}}\left(r_\alpha^{2-n}+d_g(x_\alpha,x)^{3-n} \right)+\varepsilon_\alpha B_\alpha(x),
\end{split}
\end{equation}
for all $x\in B_{x_\alpha}\left(\frac{R}{2}r_\alpha\right)\backslash\{x_\alpha\}$ 
and all $\alpha$, where $\mu_\alpha$ is as in \eqref{mu}, $r_\alpha$ as in \eqref{eq40}, $(B_\alpha)_\alpha$ as in 
\eqref{bulle}, and $\varepsilon_\alpha\to 0$ as $\alpha\to +\infty$.
\end{prop}

We refer to Hebey \cite{HebeyBookZurich} for the proofs of Lemma \ref{l61} and Proposition \ref{p61}. Now we aim in applying 
to our situation the classical scheme for stability, as developed in Druet \cite{Dcompactness} (see also Hebey \cite{HebeyBookZurich}), but face the serious difficulty that, contrary to what is a priori required 
by this scheme, and as discussed in Section \ref{NSCLPC}, we 
do not have any $C^1$-control on the $v_\alpha$'s, and thus on the $h_\alpha$'s in \eqref{SchroSingleEqt}. The loss of sufficiently smooth control 
may of course reverse the stability issue and transform an a priori stable situation into an unstable situation 
(see Druet and Laurain \cite{DruetLaurain} or Druet, Hebey and Laurain \cite{DHL} for results in this direction). We settle this loss of control on the derivatives in the next section thanks to 
the very special form of the equations in \eqref{kg} and the idea developed in Lemma \ref{controlv2}.

 \section{Blow-up theory in the critical case - Sharp asymptotics}\label{BlUpThrySharpAsy}

 Once more we let $(M,g)$ be a closed Riemannian $n$-manifold, $n\geq 5$, $m_0,m_1, q>0$, and $(\omega_\alpha)_\alpha$ be a 
 sequence in $(-m_0,m_0)$ such that $\omega_\alpha\to \omega$ as $\alpha\to +\infty$ for some $\omega\in [-m_0,m_0]$. Also, we let 
 $((u_\alpha,v_\alpha))_\alpha$ be a sequence of smooth positive functions satisfying \eqref{kgstab} for all $\alpha$, and satisfying \eqref{maxu}. At last 
 we let the $x_\alpha$'s and $\rho_\alpha$'s be points in $M$ and positive real numbers which satisfy \eqref{hyp}. We continue the a priori 
 asymptotic analysis of the blowing-up sequence $(u_\alpha)_\alpha$ around the theoretical concentration points $x_\alpha$ we 
 initiated in the preceding section. We let $\mu_\alpha$ be given by \eqref{mu} and define $\hat{v}_\alpha$ to be such that
 $$\hat{v}_\alpha(x)=v_\alpha(\exp_{x_\alpha}(\mu_\alpha x))$$
 for $x\in B_0(\rho_\alpha/\mu_\alpha)\subset \mathbb{R}^n$. By Lemma \ref{l61}, we can apply Lemma \ref{controlv2}. In particular we get that the following 
 result holds true.
 
  \begin{lem}\label{controlv}
 For any $p\in[1,+\infty[$,
 \begin{equation}\label{eqcontrolv}
 \hat{v}_\alpha\to \frac{1}{q} ~~\text{in}~L^p_{loc}(\mathbb{R}^n)
 \end{equation}
 and a.e. as $\alpha \to +\infty$.
 \end{lem}
 
 Now  we aim to get asymptotic formulas for the $u_\alpha$'s at the scale of the $r_\alpha$'s. Here we use assumption \eqref{prel1} for the first time. The following proposition 
 is the key result of this section.
 
 \begin{prop}\label{p64}
 Let $(M,g)$ be a closed Riemannian $n$-manifold, $n\geq 5$, and $((u_\alpha,v_\alpha))_\alpha$ be a sequence of smooth positive solutions of \eqref{kgstab} 
 such that \eqref{maxu} holds true. We assume \eqref{prel1} holds true everywhere in $M$. 
 Let $(x_\alpha)_\alpha$ and $(\rho_\alpha)_\alpha$ be such that \eqref{hyp} holds true. Then $r_\alpha\to 0$ as $\alpha\to +\infty$,
 \begin{equation}\label{eqp642}
 r_\alpha= \rho_\alpha
 \end{equation}
 for all $\alpha$, and 
 \begin{equation}\label{eqp64}
 r_\alpha^{n-2}\mu_\alpha^{-\frac{n-2}{2}}u_\alpha\left(\exp_{x_\alpha}(r_\alpha x) \right)\to \frac{(n(n-2))^{\frac{n-2}{2}}}{|x|^{n-2}}+\mathcal{H}(x)
 \end{equation}
 in $C^1_{loc}\left(B_0(2)\backslash \{0\}\right)$ as $\alpha\to +\infty$, 
 where $\mu_\alpha$ is as in \eqref{mu}, $r_\alpha$ as in \eqref{eq40}, and $\mathcal{H}$ is a harmonic function in $B_0(2)$ which satisfies that $\mathcal{H}(0)\leq 0$.
\end{prop} 

In order to prove Proposition \ref{p64}, we let $X_\alpha$ be the $1$-form given by
\begin{equation}\label{eqX}
X_\alpha(x)=\left(1-\frac{1}{6(n-1)}Rc_g^{\sharp}(x)(\nabla f_\alpha(x),\nabla f_\alpha (x)) \right) \nabla f_\alpha(x)
\end{equation}
 for $x\in M$, where $f_\alpha(x)=\frac{1}{2}d_g(x_\alpha,x)^2$, $Rc_g$ is the Ricci curvature tensor of $g$, and $\sharp$ is the musical isomorphism. 
 As is easily checked, the following estimates hold true 
 \begin{equation}\label{estimX}
 \begin{split}
& |X_\alpha(x) |=O(d_g(x_\alpha,x)),\\
 &\text{div}_g X_\alpha(x)=n+O(d_g(x_\alpha,x)^2),\\ 
 &\Delta_g(\text{div}_g X_\alpha)(x)=\frac{n}{n-1}S_g(x)+O(d_g(x_\alpha,x)).
 \end{split}
 \end{equation}
We define
 \begin{equation}\label{eq11}
  \tilde{\mathcal{R}}_{\alpha}= \int_{B_{x_\alpha}(r_\alpha)} (qv_\alpha-1)^2 \left\{u_\alpha X_\alpha(\nabla u_\alpha) +\frac{n-2}{2n}(\text{div}_g X_\alpha) u_\alpha^2\right\} dv_g,
  \end{equation}
  where, for a $1$-form $X$ and a smooth function $u$, $X(\nabla u) = (X,\nabla u) = g^{ij}X_i\nabla_ju$ in local coordinates. 
 The Riemannian Pohozaev identity given in Hebey \cite{HebeyBookZurich}, when applied to $u_\alpha$ in $B_{x_\alpha}(r_\alpha)$, gives that
\begin{equation}\label{eq112}
\begin{split}
&m_0^2\int_{B_{x_\alpha}(r_\alpha)} \left\{u_\alpha X_\alpha(\nabla u_\alpha) +\frac{n-2}{2n}(\text{div}_g X_\alpha) u_\alpha^2\right\} dv_g \\
&+\frac{n-2}{4n}  \int_{B_{x_\alpha}(r_\alpha)}\left(\Delta_g\text{div}_g X_\alpha \right) u_\alpha^2 dv_g  = \tilde{\mathcal{R}}_\alpha+Q_{1,\alpha}-Q_{2,\alpha}+Q_{3,\alpha},
\end{split}
\end{equation}
where $\tilde{\mathcal{R}}_\alpha$ is as in \eqref{eq11},
\begin{equation}\label{DefQIALpha}
\begin{split}
Q_{1,\alpha}=&\frac{n-2}{2n}\int_{\partial B_{x_\alpha}(r_\alpha)} (\text{div}_g X_\alpha)(\partial_\nu u_\alpha) u_\alpha d\sigma_g\\
&-\int_{\partial B_{x_\alpha}(r_\alpha)} \left(\frac{1}{2}X_\alpha(\nu) |\nabla u_\alpha|^2 - X_\alpha(\nabla u_\alpha) \partial_\nu u_\alpha \right) d\sigma_g,\\
Q_{2,\alpha}= & \int_{B_{x_\alpha}(r_\alpha)}\left(\nabla X_\alpha-\frac{1}{n}(\text{div}_g X_\alpha)g \right)^\sharp(\nabla u_\alpha, \nabla u_\alpha) dv_g,\\
Q_{3,\alpha}=& \frac{n-2}{2n} \int_{\partial B_{x_\alpha}(r_\alpha)} X_\alpha(\nu) u_\alpha^{2^\star} d\sigma_g\\
&-\frac{n-2}{4n}\int_{\partial B_{x_\alpha}(r_\alpha)}\left(\partial_\nu (\text{div}_g X_\alpha) \right) u_\alpha^2 d\sigma_g,
\end{split}
\end{equation}
and, in the above expressions, $\nu$ is the unit outward normal to $\partial B_{x_\alpha}(r_\alpha)$. We need two intermediate lemmas before proving Proposition \ref{p64}. The 
first lemma is as follows.

 \begin{lem}\label{crucial}
 Let $(M,g)$ be a closed Riemannian $n$-manifold, $n\geq 5$, and $((u_\alpha,v_\alpha))_\alpha$ be a sequence of smooth positive solutions of \eqref{kgstab} 
 such that \eqref{maxu} holds true. Let $(x_\alpha)_\alpha$ and $(\rho_\alpha)_\alpha$ be such that \eqref{hyp} holds true.  Let $r_\alpha$ be as in \eqref{eq40}, $\mu_\alpha$ as in \eqref{mu}, $X_\alpha$ as in \eqref{eqX}, and $\bar{x}_0$ be such that, 
 up to a subsequence, $x_\alpha\to \bar{x}_0$ as $\alpha\to +\infty$. Let $\tilde{\mathcal{R}}_{\alpha}$ be as in \eqref{eq11}. 
Then there holds that 
 \begin{equation}\label{eqcrucial}
 \tilde{\mathcal{R}}_{\alpha}=o(\mu_\alpha^2)
  \end{equation}
 and we also have that
\begin{equation}\label{eqcrucial2}
\begin{split}
& \int_{B_{x_\alpha}(r_\alpha)} \left\{u_\alpha X_\alpha(\nabla u_\alpha) +\frac{n-2}{2n}(\text{div}_g X_\alpha) u_\alpha^2\right\} dv_g = \mu_\alpha^{2}\left(-C_n+o(1) \right),\\
& \int_{B_{x_\alpha}(r_\alpha)}\left(\Delta_g\text{div}_g X_\alpha \right) u_\alpha^2 dv_g = \mu_\alpha^{2} \left(\frac{n S_g(\bar{x}_0)}{n-1}C_n+o(1)\right),
\end{split}
\end{equation}   
where $C_n=\int_{\mathbb{R}^n} u_0^2 dx$, $u_0$ is as in \eqref{u0}, and $S_g$ is the scalar curvature of $g$.
\end{lem}

\begin{proof}[Proof of Lemma \ref{crucial}] For $x\in B_0(\rho_\alpha/\mu_\alpha)\subset \mathbb{R}^n$ we let 
 \begin{equation}\label{eq12}
 \begin{split}
 &\hat{u}_\alpha(x ) =  u_\alpha\left(\exp_{x_\alpha}(\mu_\alpha x)\right),\\
 &  \hat{v}_\alpha(x) =  v_\alpha\left(\exp_{x_\alpha}(\mu_\alpha x)\right),\\
 &\hat{\varphi}_\alpha(x) =  {\text{div}_g}(X_\alpha)\left(\exp_{x_\alpha}(\mu_\alpha x)\right),\\
& \hat{\psi}_\alpha(x)  =  \left(X_\alpha(\nabla u_\alpha)\right)\left(\exp_{x_\alpha}(\mu_\alpha x)\right),\\
&\hat{g}_\alpha(x) = \left(\exp^\star_{x_\alpha} g\right)(\mu_\alpha x).
 \end{split}
 \end{equation}
 Also we define
 \begin{equation}\label{eq12bis}
 \hat{\Psi}_\alpha(x)  =  \left(\mu_\alpha^{\frac{n-2}{2}} \hat{u}_\alpha (x) \right) \mu_\alpha^{\frac{n-2}{2}} \hat{\psi}_\alpha(x) 
 +\frac{n-2}{2n}\hat{\varphi}_\alpha(x) \left(\mu_\alpha ^{\frac{n-2}{2}} \hat{u}_\alpha (x) \right)^2.
 \end{equation}
 Since $\mu_\alpha\to 0$, we get that $\hat{g}_\alpha \to \xi$ in $C^2_{loc}(\mathbb{R}^n)$ as $\alpha\to +\infty$, 
 where $\xi$ is the Euclidean metric in $\mathbb{R}^n$. On the other hand, thanks to \eqref{eql61}, \eqref{eq41} and \eqref{estimX}, we can write that
 \begin{equation}\label{eq13}
 \hat{\Psi}_\alpha\to \left(u_0 \langle \nabla u_0,x\rangle_{\xi}+\frac{n-2}{2}u_0^2 \right) \text{ in } C^0_{loc}(\mathbb{R}^n),
 \end{equation}
 where $u_0$ is as in \eqref{u0}. Setting 
 \begin{equation}\label{Psi1}
 \Psi_\alpha=u_\alpha X_\alpha(\nabla u_\alpha)  + \frac{n-2}{2n}(\text{div}_g X_\alpha) u_\alpha^2,
 \end{equation}
 we get  by mixing \eqref{eqcontrolv} and  \eqref{eq13} that
 \begin{equation}\label{eq16}
 \begin{split}
 \int_{B_{x_\alpha}(\mu_\alpha)}  (qv_\alpha-1)^2 \Psi_\alpha dv_g & =  \mu_\alpha^2 \int_{B_0(1)}(q\hat{v}_\alpha-1)^2 \hat{\Psi}_\alpha dv_{\hat{g}_\alpha},\\
 & = o(\mu_\alpha^2).
 \end{split}
 \end{equation}
Using now \eqref{eqp61}, we can write that there exists  $C>0$ independent of $\alpha$ such that
\begin{equation}\label{eq14}
\left|\hat{\Psi}_\alpha(x) \right| \leq \frac{C}{|x|^{2n-4}}
\end{equation}
for all $x \in B_0\left(\frac{r_\alpha}{\mu_\alpha} \right)\backslash B_0(1)$. 
Then, since $n\geq5$ and $0\leq \hat{v}_\alpha \leq \frac{1}{q}$, we get by \eqref{eqcontrolv}, \eqref{eq13}, \eqref{eq14}, and the Lebesgue's 
dominated convergence theorem that
\begin{equation}\label{eq15}
 \begin{split}
 \int_{B_{x_\alpha}(r_\alpha)\backslash B_{x_\alpha}(\mu_\alpha)}  (qv_\alpha-1)^2 \Psi_\alpha dv_g & =  \mu_\alpha^2 \int_{B_0(r_\alpha/\mu_\alpha)\backslash B_0(1)}(q\hat{v}_\alpha-1)^2 \hat{\Psi}_\alpha dv_{\hat{g}_\alpha}\\
  & = o(\mu_\alpha^2).
  \end{split}
\end{equation}
In particular, combining \eqref{eq16} and \eqref{eq15}, we get that  \eqref{eqcrucial} holds true. 
Now, we write that
$$\int_{B_{x_\alpha}(r_\alpha)} \Psi_\alpha(x) dv_g(x)=\mu_\alpha^2 \int_{B_0\left(\frac{r_\alpha}{\mu_\alpha}\right)}\hat{\Psi}_\alpha(x) dv_{\hat{g}_\alpha}(x), $$
and thus, using  \eqref{eq41}, \eqref{eq13}, \eqref{eq14}, and arguing as above, we get that
\begin{equation}\label{crucial100}
\begin{split}
 \int_{B_{x_\alpha}(r_\alpha)} \left\{u_\alpha X_\alpha(\nabla u_\alpha) +\frac{n-2}{2n}(\text{div}_g X_\alpha) u_\alpha^2\right\} dv_g\\
   ~~= \mu_\alpha^{2}\left(\int_{\mathbb{R}^n} \left(u_0 \langle \nabla u_0,x\rangle_{\xi}+\frac{n-2}{2}u_0^2 \right) dx+o(1) \right).
  \end{split}
\end{equation}
Integrating by parts, there holds that
\begin{equation}\label{crucial101}
\int_{\mathbb{R}^n} \left(u_0 \langle \nabla u_0,x\rangle_{\xi}+\frac{n-2}{2}u_0^2 \right) dx=-\int_{\mathbb{R}^n} u_0^2 dx.
\end{equation}
By \eqref{crucial100} and \eqref{crucial101}, the first equation in \eqref{eqcrucial2} holds true. The second one holds true the same way, using \eqref{eqp61} and the 
third equation in \eqref{estimX}. This ends the proof of Lemma \ref{crucial}.
\end{proof}

The second lemma we need to prove Proposition \ref{p64} is the following. Its proof only uses the fact that the 
sequence $(h_\alpha)_\alpha$ in \eqref{h} and \eqref{SchroSingleEqt} is bounded in $L^\infty$. Then we can proceed as in Hebey \cite{HebeyBookZurich} 
to prove it. We refer to 
Hebey \cite{HebeyBookZurich} for the proof of Lemma \ref{l64}. 

\begin{lem}\label{l64}
 Let $(M,g)$ be a closed Riemannian $n$-manifold, $n\geq 5$, and $((u_\alpha,v_\alpha))_\alpha$ be a sequence of smooth positive solutions of \eqref{kgstab} 
 such that \eqref{maxu} holds true. Let $(x_\alpha)_\alpha$ and $(\rho_\alpha)_\alpha$ be such that \eqref{hyp} holds true. Let $Q_{2,\alpha}$ be as in 
 \eqref{DefQIALpha}. Then
\begin{equation}\label{eql64}
Q_{2,\alpha}=o(\mu_\alpha^2)+o(\mu_\alpha^{n-2}r_\alpha^{2-n})~,
\end{equation}
where $\mu_\alpha$ is as in \eqref{mu} and $r_\alpha$ as in \eqref{eq40}.
\end{lem}

Now that we have Lemmas \ref{crucial} and \ref{l64} we can prove Proposition \ref{p64}. 

\begin{proof}[Proof of Proposition \ref{p64}] Let $R\geq 6$ be such that $Rr_\alpha\leq 6\rho_\alpha$ for all $\alpha\gg 1$. We assume first that $r_\alpha\to 0$ as $\alpha\to +\infty$. Given $x\in B_0(R)$, we define
\begin{equation}\label{defw1}
\begin{split}
&w_\alpha(x) = r_\alpha^{n-2} \mu_\alpha^{-\frac{n-2}{2}} u_\alpha\left(\exp_{x_\alpha}(r_\alpha x) \right),\\
&g_\alpha(x) = \left(\exp_{x_\alpha}^\star g\right)(r_\alpha x),~ \text{and}\\
&\tilde{h}_\alpha(x) = h_\alpha\left(\exp_{x_\alpha}(r_\alpha x) \right),
\end{split}
\end{equation}
where $h_\alpha$ is as in \eqref{h}. Since $r_\alpha\to 0$ as $\alpha\to +\infty$, we have that $g_\alpha\to \xi$ in $C^2_{loc}(\mathbb{R}^n)$, where $\xi$ is the Euclidean metric. Thanks to Proposition \ref{p61}, we also have that
\begin{equation}\label{eq102}
|w_\alpha(x)|\leq C|x|^{2-n}
\end{equation}
in $B_0\left(\frac{R}{2}\right)\backslash \{0\}$. By \eqref{kgstab},
\begin{equation}\label{eq103}
\Delta_{g_\alpha} w_\alpha +r_\alpha^2\tilde{h}_\alpha w_\alpha=\left(\frac{\mu_\alpha}{r_\alpha}\right)^2 w_\alpha^{2^\star-1}
\end{equation}
in $B_0\left(\frac{R}{2} \right)$. Using \eqref{eq41}, the fact that $(h_\alpha)_\alpha$ is bounded in $L^\infty$, and standard elliptic theory, we get that, 
after passing to a subsequence, 
\begin{equation}\label{eq104}
w_\alpha\to w \quad \text{in}\quad C^1_{loc}\left(B_0\left(R/2\right)\backslash \{0\}\right)
\end{equation}
as $\alpha \to +\infty$, where $w$ is non-negative and harmonic in $B_0\left(\frac{R}{2}\right)\backslash \{0\}$. Thanks to \eqref{eq102}, we have that
\begin{equation}\label{eq105}
|w(x)|\leq C|x|^{2-n}
\end{equation}
in $B_0\left(\frac{R}{2}\right)\backslash \{0\}$. Thus, thanks to the B\^ocher theorem on singularities of non-negative harmonic functions, we can write that
\begin{equation}\label{eq106}
w(x)= \frac{\Lambda}{|x|^{n-2}}+\mathcal{H}(x),
\end{equation}
where $\Lambda\geq 0$ and $\mathcal{H}$ is harmonic in the full ball, namely satisfies $\Delta \mathcal{H}=0$ in $B_0\left(\frac{R}{2} \right)$. In order to see that 
$\Lambda=(n(n-2))^{\frac{n-2}{2}}$, we integrate \eqref{eq103} in $B_0(1)$ and proceed as in Hebey \cite{HebeyBookZurich}. 
Now we prove that $\mathcal{H}(0)\leq 0$ and that $r_\alpha\to 0$ as $\alpha \to +\infty$. For that purpose, we return to the Riemannian 
Pohozaev identity \eqref{eq112}. By  \eqref{eq41} and \eqref{estimX}, 
\begin{equation}\label{eq113}
Q_{3,\alpha}=O(\mu_\alpha^n r_\alpha^{-n})+O(\mu_\alpha^{n-2}r_\alpha^{4-n}) = o(\mu_\alpha^{n-2}r_\alpha^{2-n})+o(\mu_\alpha^2).
\end{equation}
By \eqref{eql64}, $Q_{2,\alpha}$ is known. By Lemma \ref{crucial}, the left hand side in \eqref{eq112} and $\tilde{\mathcal{R}}_\alpha$ are known. 
In particular, we get with \eqref{eq112}, \eqref{eql64}, \eqref{eq113}, and Lemma \ref{crucial} that
\begin{equation}\label{eq114}
Q_{1,\alpha}=\left(\int_{\mathbb{R}^n}u_0^2 dx\right) \left(\frac{n-2}{4(n-1)}S_g(\bar{x}_0)-m_0^2 \right)\mu_\alpha^2+o(\mu_\alpha^2)+o(\mu_\alpha^{n-2}r_\alpha^{2-n}),
\end{equation}
where $x_\alpha\to\bar{x}_0$, as $\alpha\to +\infty$. By \eqref{eqp61}, \eqref{estimX} and the expression of $Q_{1,\alpha}$, we have that
\begin{equation}\label{RoughEstQ1ALpha}
Q_{1,\alpha}=O(\mu_\alpha^{n-2}r_\alpha^{2-n}) .
\end{equation}
By the assumption \eqref{prel1} of Theorem \ref{CritThm}, since $n \ge 5$, we get from \eqref{eq114} and \eqref{RoughEstQ1ALpha} that
\begin{equation}\label{eq115}
r_\alpha \to 0
\end{equation}
as $\alpha\to +\infty$. Then \eqref{eq104}, \eqref{eq105} and \eqref{eq106} hold true and, as one can check, 
we get in turn that
\begin{equation}\label{eq116}
Q_{1,\alpha}=-\left(\frac{1}{2} {(n-2)^2} \omega_{n-1} \Lambda \mathcal{H}(0) +o(1) \right) \mu_\alpha^{n-2}r_\alpha^{2-n}.
\end{equation}
Coming back to \eqref{eq114}, it follows from \eqref{eq116} that
\begin{equation}\label{eq117}
\frac{1}{2} {(n-2)^2} \omega_{n-1} \Lambda \mathcal{H}(0)=-C_n \left(\frac{n-2}{4(n-1)}S_g(\bar{x}_0)-m_0^2 \right)\lim_{\alpha\to +\infty} \left(\mu_\alpha^{4-n}r_\alpha^{n-2} \right),
\end{equation}
where $C_n=\int_{\mathbb{R}^n}u_0^2 dx$. 
Using again the assumption \eqref{prel1} of Theorem \ref{CritThm} we get that
\begin{equation}\label{eq118}
\mathcal{H}(0)\leq 0.
\end{equation}
At this point, it remains to prove that $\rho_\alpha=r_\alpha$. If it is not the case, then 
$r_\alpha<\rho_\alpha$ and we get with \eqref{eq43} that $\left(r^{\frac{n-2}{2}} \varphi(r)\right)'(1) = 0$, where
$$\varphi(r)=\frac{1}{\omega_{n-1} r^{n-1}}\int_{\partial B_0(r)} w d\sigma=\frac{(n(n-2))^{\frac{n-2}{2}}}{r^{n-2}}+\mathcal{H}(0),$$
by \eqref{eq104}, and \eqref{eq106}. Hence $\mathcal{H}(0)=(n(n-2))^{\frac{n-2}{2}}>0$ and we get a contradiction with \eqref{eq118}. 
This ends the proof of Proposition \ref{p64}.
\end{proof}

 \section{A priori bounds in the critical case}\label{APriBdCritCase}
 
 We prove the a priori bound property in Theorem \ref{CritThm}. We let $(M,g)$ be a smooth 
 closed Riemannian $n$-manifold, $n\geq 5$, $m_0,m_1, q>0$ and $(\omega_\alpha)_\alpha$ be a 
 sequence in $(-m_0,m_0)$ such that $\omega_\alpha\to \omega$ as $\alpha\to +\infty$ for some $\omega\in [-m_0,m_0]$. Also, we let 
 $((u_\alpha,v_\alpha))_\alpha$ be a sequence of smooth positive functions satisfying \eqref{kgstab} for all $\alpha$, and satisfying \eqref{maxu}. We assume 
 that \eqref{prel1} of Theorem \ref{CritThm} holds true everywhere in $M$. At this point, the sequence $(u_\alpha)_\alpha$ is not bounded in $H^1$, so we cannot apply the 
 $H^1$-theory for blow-up. Instead, by Lemma 6.7 in Hebey \cite{HebeyBookZurich}, we get that there exists $C_1>0$ such that for any $\alpha$, there exist $N_\alpha\in\mathbb{N}^\star$ and $N_\alpha$ critical points of $u_\alpha$, denoted by $\left\{x_{1,\alpha},...,x_{N_\alpha,\alpha} \right\}$ such that 
\begin{equation}\label{l671}
d_g(x_{i,\alpha},x_{j,\alpha})^{\frac{n-2}{2}} u_\alpha(x_{i,\alpha})\geq 1
\end{equation} 
for all $i,j\in \left\{1,...,N_\alpha \right\}$, $i\neq j$, and 
\begin{equation}\label{l672}
\left(\min_{i=1,...,N_\alpha} d_g(x_{i,\alpha},x) \right)^{\frac{n-2}{2}} u_\alpha(x) \leq C_1
\end{equation}  
for all $x\in M$ and all $\alpha$. We define \begin{equation}\label{eq129}
d_\alpha=\min_{1\leq i<j \leq N_\alpha} d_g(x_{i,\alpha},x_{j,\alpha}).
\end{equation}
If $N_\alpha=1$, we set $d_\alpha=\frac{1}{4}i_g$, where $i_g$ is the injectivity radius of $(M,g)$. In case $N_\alpha\geq 2$, we reorder the $x_{i,\alpha}$'s such that
\begin{equation}\label{eq130}
d_\alpha=d_g(x_{1,\alpha},x_{2,\alpha})\leq d_g(x_{1,\alpha},x_{3,\alpha}) \leq ... \leq d_g(x_{1,\alpha},x_{N_\alpha,\alpha}).
\end{equation}
We prove now that the $d_\alpha$'s do not converge to zero and thus that blow-up points have to be isolated when we assume \eqref{prel1}.

\begin{prop}\label{p65}
 Let $(M,g)$ be a closed Riemannian $n$-manifold, $n\geq 5$, and $((u_\alpha,v_\alpha))_\alpha$ be a sequence of smooth positive solutions of \eqref{kgstab} 
 such that \eqref{maxu} holds true. We assume that \eqref{prel1} holds true. Then, up to a subsequence, the sequence $(d_\alpha)_\alpha$ converges to a positive constant, 
 where $d_\alpha$ is as in \eqref{eq129}.
\end{prop}

\begin{proof}[Proof of Proposition \ref{p65}] We proceed by contradiction and assume that $d_\alpha\to 0$ as $\alpha\to +\infty$. 
Then, $N_\alpha\geq 2$ for $\alpha$ large. We assume that the concentration points are ordered in such a way that \eqref{eq130} holds true. We set for $x\in B_0(\delta d_\alpha^{-1})$, $0<\delta< \frac{1}{2}i_g$ fixed,
\begin{equation}\label{eq131}
\begin{split}
&\check{u}_\alpha(x) = d_\alpha^{\frac{n-2}{2}} u_\alpha\left(\exp_{x_{1,\alpha}} (d_\alpha x) \right),\\
&\check{h}_\alpha(x) = h_\alpha\left(\exp_{x_{1,\alpha}}(d_\alpha x) \right),\\
&\check{g}_\alpha(x) = \left(\exp^\star_{x_{1,\alpha}} g \right) (d_\alpha x),
\end{split}
\end{equation}
where $h_\alpha$ is as in \eqref{h}.
We have $\check{g}_\alpha \to \xi$ in $C^{2}_{loc}(\mathbb{R}^n)$ as $\alpha\to +\infty$ since $d_\alpha \to 0$. Thanks to the first equation in \eqref{kgstab}, we 
also have that
\begin{equation}\label{eq132}
\Delta_{\check{g}_\alpha} \check{u}_\alpha +d_\alpha^2 \check{h}_\alpha \check{u}_\alpha =\check{u}_\alpha^{2^\star-1}
\end{equation}
in $B_0(\delta d_\alpha^{-1})$. For any $R>0$, we let $1\leq N_{R,\alpha} \leq N_\alpha$ be such that
\begin{equation*}
\begin{split}
&d_g(x_{1,\alpha},x_{i,\alpha})\leq R d_\alpha ~\text{for}~ 1\leq i\leq N_{R,\alpha}, \text{ and}\\
&d_g(x_{1,\alpha},x_{i,\alpha}) > R d_\alpha ~\text{for}~ N_{R,\alpha}+1\leq i\leq N_\alpha.
\end{split}
\end{equation*}
Such an $N_{R,\alpha}$ does exist thanks to \eqref{eq130}. We also have that $N_{R,\alpha}\geq 2$ for all $R>1$ and that $(N_{R,\alpha})_\alpha$ is bounded for all $R>0$ thanks to \eqref{eq129}. Indeed, suppose that there are $k_\alpha$ points $x_{i,\alpha}$, $i=1,...,k_\alpha$, such that $d_g(x_{1,\alpha},x_{i,\alpha})\leq R d_\alpha$. By \eqref{eq129} 
$$B_{x_{i,\alpha}}\left(\frac{d_\alpha}{2}\right)\cap B_{x_{j,\alpha}}\left(\frac{d_\alpha}{2} \right)=\emptyset $$
for all $i\neq j$. Then, 
$$\text{Vol}_g\left(B_{x_{1,\alpha}}\left(\frac{3R}{2} d_\alpha\right)\right)\geq \sum_{i=1}^{k_\alpha} \text{Vol}_g\left(B_{x_{i,\alpha}}\left(\frac{d_\alpha}{2} \right) \right) $$
and we get an upper bound for $k_\alpha$ depending only on $R$. In the sequel, we set
$$\check{x}_{i,\alpha}=d_\alpha^{-1} \exp_{x_{1,\alpha}}^{-1}(x_{i,\alpha}) $$
for all $1\leq i\leq N_\alpha$ such that $d_g(x_{1,\alpha},x_{i,\alpha})\leq \frac{i_g}{2}$. Thanks to \eqref{l672}, for any $R>1$, there exists $C_R>0$ such that
\begin{equation}\label{eq133}
\underset{\Omega_{R,\alpha}}{\sup} \check{u}_\alpha \leq C_R,
\end{equation}
where 
$$\Omega_{R,\alpha}=B_0(R)\backslash \cup_{i=1}^{N_{2R,\alpha}}B_{\check{x}_{i,\alpha}}\left(\frac{1}{R} \right).$$
As in Hebey \cite{HebeyBookZurich}, one easily gets that for any $R>1$, there exists $D_R>0$ such that
\begin{equation}\label{eq134} 
\|\nabla \check{u}_\alpha \|_{L^\infty(\Omega_{R,\alpha})}\leq D_R \underset{\Omega_{R,\alpha}}{\sup} \check{u}_\alpha \leq D_R^2 \underset{\Omega_{R,\alpha}}{\inf} \check{u}_\alpha.
\end{equation}
Assume first that, for some $R>0$, there exists $1\leq i \leq N_{R,\alpha}$ such that 
\begin{equation}\label{eq135}
\check{u}_\alpha(\check{x}_{i,\alpha})=O(1).
\end{equation}
Since the first two equations of \eqref{hyp} are satisfied by the sequences $x_\alpha=x_{i,\alpha}$ and $\rho_\alpha=\frac{1}{8} d_\alpha$, it follows from Lemma \ref{l61} that the last equation in \eqref{hyp} cannot hold and thus that $(\check{u}_\alpha)_\alpha$ is uniformly bounded in $B_{\check{x}_{i,\alpha}}\left(\frac{3}{4} \right)$. In particular, by standard elliptic theory, thanks to \eqref{eq132} and the sequence $(h_\alpha)_\alpha$ being bounded in $L^\infty$, $(\check{u}_\alpha)_\alpha$ is uniformly bounded in $C^1(B_{\check{x}_{i,\alpha}}\left(\frac{1}{2} \right))$. By \eqref{l671}, assuming $i \not= 1$, 
we have that
$$|\check{x}_{i,\alpha}|^{\frac{n-2}{2}}\check{u}_\alpha(\check{x}_{i,\alpha})\geq 1$$
and we get the existence of some $\delta_i>0$ such that
$$\check{u}_\alpha\geq \frac{1}{2}|\check{x}_{i,\alpha}|^{1-\frac{n}{2}}\geq \frac{1}{2} R^{1-\frac{n}{2}} $$
in $B_{\check{x}_{i,\alpha}}(\delta_i)$. If $i=1$, applying \eqref{l671} with $i=1$ and $j=2$, we get that $\check{u}_\alpha(\check{x}_{1,\alpha}) \ge 1$, and the above inequality 
remains true for $R > 1$.
Assume now that, for some $R>0$, there exists $1\leq i\leq N_{R,\alpha}$ such that
\begin{equation}\label{eq136}
\check{u}_\alpha(\check{x}_{i,\alpha})\to +\infty
\end{equation}
as $\alpha\to +\infty$. Then, \eqref{hyp} is satisfied by the sequence $x_\alpha=x_{i,\alpha}$ and $\rho_\alpha=\frac{1}{8} d_\alpha$, and it follows from Proposition \ref{p64} that the sequence $(\check{u}_\alpha(\check{x}_{i,\alpha})\check{u}_\alpha)_\alpha$ is bounded in 
$$\check{\Omega}_{\alpha}=B_{\check{x}_{i,\alpha}}(\tilde{\delta}_i)\backslash B_{\check{x}_{i,\alpha}}(\frac{\tilde{\delta}_i}{2})
$$
for some $\tilde{\delta}_i>0$. Thus, using the Harnack type part of \eqref{eq134}, we can deduce that these two situations are mutually exclusive in the sense that either \eqref{eq135} holds true for all $i$ or \eqref{eq136} holds true for all $i$. We can thus split the conclusion of the proof into two cases. 
In case 1 : we assume that there exist $R>0$ and $1\leq i \leq N_{R,\alpha}$ such that $\check{u}_\alpha(\check{x}_{i,\alpha})=O(1)$. Then, thanks to the above discussion, we get that
$$\check{u}_\alpha(\check{x}_{j,\alpha})=O(1) $$
for all $1\leq j \leq N_{R,\alpha}$ and all $R>0$. Arguing as above and using \eqref{eq133} and \eqref{eq134}, it follows that $(\check{u}_\alpha)_\alpha$ is uniformly bounded in $C^1_{loc}(\mathbb{R}^n)$. Thus, by \eqref{eq132} and by 
elliptic theory, there exists a subsequence of $(\check{u}_\alpha)_\alpha$ which converges in $C^1_{loc}(\mathbb{R}^n)$ to some $\check{u}$ solution of 
$$\Delta \check{u}=\check{u}^{2^\star -1} $$
in $\mathbb{R}^n$. Still thanks to the above discussion, we know that $\check{u}\not\equiv 0$. Moreover, $\check{u}$ possesses at least two critical points, namely $0$ and some $\check{x}_{2}$, 
$\vert\check{x}_{2}\vert = 1$, 
limit of a subsequence of $(\check{x}_{2,\alpha})_\alpha$. This is absurd thanks to the classification result of Caffarelli, Gidas and Spruck \cite{CaffarelliGidas}. 
In case 2 : we assume that there exist $R>0$ and $1\leq i\leq N_{R,\alpha}$ such that $\check{u}_\alpha(\check{x}_{i,\alpha})\to +\infty$ as $\alpha\to +\infty$. Then, thanks to the above discussion, we get that $\check{u}_\alpha(\check{x}_{j,\alpha})\to +\infty$ as $\alpha\to +\infty$, for all $1\leq j\leq N_{R,\alpha}$ and all $R>0$. By \eqref{eq132}, we have that 
$$\Delta_{\check{g}_\alpha} \check{v}_\alpha+ d_\alpha^2 \check{h}_\alpha \check{v}_\alpha =\frac{1}{|\check{u}_\alpha(0)|^{2^\star-2}} \check{v}_\alpha^{2^\star -1},$$
where $\check{v}_\alpha=\check{u}_\alpha(0) \check{u}_\alpha$. Applying Proposition \ref{p64} with the above discussion on the one hand, and standard elliptic theory with \eqref{eq134} on the other hand, we get from the preceding equation for $\check{v}_\alpha$ that, after passing to a subsequence,
$$\check{u}_\alpha(0) \check{u}_\alpha\to \check{G} $$
in $C^1_{loc}(\mathbb{R}^n\backslash \{\check{x}_i \}_{i\in I})$ as $\alpha\to +\infty$, where 
$$I=\left\{1,...,\lim_{R\to +\infty}\lim_{\alpha\to +\infty} N_{R,\alpha}\right\} $$
and, for any $R>0$,
$$\check{G}(x)=\sum_{i=1}^{\tilde{N}_R} \frac{\tilde{\Lambda}_i}{|x-\check{x}_{i}|^{n-2}}+\check{H}_R(x) $$
in $B_{0}(R)$, where $1\leq \tilde{N}_R\leq N_{2R}$ is such that $|\check{x}_{\tilde{N}_R}|\leq R$ and $|\check{x}_{\tilde{N}_R+1}|> R$, and where $N_{2R,\alpha}\to N_{2R}$ as $\alpha\to +\infty$. Here, up to a subsequence, we assume $\check{x}_{i,\alpha}\to \check{x}_i$ as $\alpha\to +\infty$, for all $i\in I$. By \eqref{eq134},
$$\frac{\check{u}_\alpha(\check{x}_{i,\alpha})}{\check{u}_\alpha(0)} \to \mu_i$$
as $\alpha \to +\infty$ for some $\mu_i > 0$. In particular, 
the $\tilde{\Lambda}_i$'s are positive real numbers and $\check{H}_R$ is a harmonic function in $B_0(R)$. We have that
$$\check{H}_{R_1}(x)-\check{H}_{R_2}(x)=\sum_{i=\tilde{N}_{R_1}+1}^{\tilde{N}_{R_2}}\frac{\tilde{\Lambda}_i}{|x-\check{x}_i|^{n-2}} $$
for all $0<R_1<R_2$. We can write that
$$\check{G}(x)=\frac{\tilde{\Lambda}_1}{|x|^{n-2}}+X(x) $$
where for any $R>1$ and any $x\not\in\{\check{x}_i, i\in I\}$,
$$X(x) = \sum_{i=2}^{\tilde{N}_R} \frac{\tilde{\Lambda}_i}{|x-\check{x}_i|^{n-2}}+\check{H}_R(x).$$
 For $R>1$, we set $\Omega_{R,\tilde{\gamma}}=B_0(R)\backslash \cup_{i=2}^{\check{N}_R}B_{\check{x}_{i}}(\tilde{\gamma})$. 
 Since $R$ is fixed, we can choose $0<\tilde{\gamma}\ll 1$ such that $X$ attains its minimum over $\partial \Omega_{R,\tilde{\gamma}}$ in $\partial B_0(R)$. Hence, using the maximum principle for the harmonic function $X$ in $\Omega_{R,\tilde{\gamma}}$ and that $\check{G}\geq 0$, we get $X(0)\geq -\tilde{\Lambda}_1 R^{2-n}$. Since $R > 1$ is arbitrary, $X(0)\geq 0$. By Proposition \ref{p64}, we now get that $X(0)=0$.  For $R>1$ fixed, we let  $\tilde{\Omega}_{R,\tilde{\gamma}}=B_0(R)\backslash \cup_{i=3}^{\check{N}_R}B_{\check{x}_{i}}(\tilde{\gamma})$. We can choose again $0<\tilde{\gamma}\ll 1$ such that $X-\tilde{\Lambda}_2/|.-\check{x}_2|^{n-2}$ attains its minimum over $\partial\tilde{\Omega}_{R,\tilde{\gamma}}$ in $\partial B_0(R)$. Hence, using the maximum principle for the harmonic function $X-\tilde{\Lambda}_2/|.-\check{x}_2|^{n-2}$ in $\tilde{\Omega}_{R,\tilde{\gamma}}$ and that $\check{G}\geq 0$, we get that
 $$X(0)\geq \tilde{\Lambda}_2-\frac{\tilde{\Lambda}_1}{R^{n-2}}-\frac{\tilde{\Lambda}_2}{(R-1)^{n-2}}. $$
 Choosing $R\gg 1$ sufficiently large, we get that $X(0)>0$ and this is in contradiction with $X(0)=0$. This ends the proof of the Proposition \ref{p65}.
 \end{proof}

We are now in position to prove the a priori bound property in Theorem \ref{CritThm}. This is the subject of what follows.

\begin{proof}[Proof of the a priori bound property in Theorem \ref{CritThm}.] By Druet and Hebey \cite{kgmp3}, and Hebey and Truong \cite{kgmp4}, we only need 
to address the cases $n \ge 5$. 
Then the proof mixes the results of Propositions \ref{p64} and \ref{p65}. By Proposition \ref{p65}, $M$ being compact, $(N_\alpha)_\alpha$ is a bounded sequence. Up to a subsequence, we can assume that $N_\alpha=N$ for all $\alpha$ and some $N\in \mathbb{N}^\star$. Let $(x_\alpha)_\alpha$ be a sequence of maximal points of $u_\alpha$. By \eqref{maxu}, $u_\alpha(x_\alpha)\to +\infty$ as $\alpha\to +\infty$ and we get by \eqref{l672} that $d_g(x_{i,\alpha},x_\alpha)\to 0$ as $\alpha\to +\infty$ for some $i$. Then, by Proposition \ref{p65}, noting that
\begin{equation}\label{end1}
d_g(x_\alpha,x)u_\alpha^{\frac{2}{n-2}}(x) \le d_g(x_\alpha,x_{i,\alpha})u_\alpha^{\frac{2}{n-2}}(x_\alpha) + d_g(x_{i,\alpha},x)u_\alpha^{\frac{2}{n-2}}(x)~,
\end{equation}
we get that \eqref{hyp} holds true with the $x_\alpha$'s and $\rho_\alpha=\bar{\delta}$, for some $\bar{\delta}>0$. But this contradicts Proposition \ref{p64} for which $\rho_\alpha\to 0$ as $\alpha\to+\infty$. In particular, there does not exists a sequence $((u_\alpha,v_\alpha))_\alpha$ of smooth positive solutions of \eqref{kgstab} such that \eqref{maxu} holds true. Standard elliptic theory concludes the proof of 
 the a priori bound property in Theorem \ref{CritThm}.
\end{proof}

 \section{A remark on \eqref{prel1}}\label{RkPotentials}

The assumption \eqref{prel1} in Theorem \ref{CritThm} is sharp in the following sense : if we take for $(M,g)$ the standard sphere $S^n$ and if we consider the 
first equation of \eqref{kg} with $\omega=0$ and $m_0^2=\frac{n(n-2)}{4}=\frac{(n-2)}{4(n-1)}S_g$, then \eqref{prel1} is not satisfied, and 
there are indeed blowing-up solutions $(u_\alpha,v_\alpha)$ for the system, where $u_\alpha$ is given by \eqref{bullesphere} and $v_\alpha=\Phi(u_\alpha)$, with $\Phi$ as in \eqref{eqphi}. Yet, we have to be careful generally speaking about the necessity of an assumption telling that the potential needs to sit below the geometric threshold $\frac{n-2}{4(n-1)} S_g$ to get the existence of positive solutions and a priori bounds, when this potential depends on the solution, as in this paper. This is stated more precisely in the following proposition, which is proved in this section. 

\begin{prop}\label{necess}
  Let $(M,g)$ be a closed Riemannian $n$-manifold ($n\geq 5$) and $q$, $m_0$, $m_1>0$, $\theta\in(0,1)$ be given. We can build a continuous map $\bar{h} : \mathbb{R}\times C^{1,\theta}\to C^{0,\theta}$ satisfying that
  \begin{equation*}
  \begin{split}
&(i)~ \bar{h}_\omega(u) > \frac{n-2}{4(n-1)} S_g~\hbox{somewhere in}~M~\hbox{for all}~\omega~\text{and all}~u,\\
&(ii)~  \bar{h}_\omega(u)~\hbox{is nonnegative nonzero for all}~\omega~\text{and all}~u,\\
 &(iii) ~\bar{h}~\text{is uniformly bounded}\hbox{ in subsets}~K\times C^{1,\theta}~\text{for}~K\subset \mathbb{R}~\text{compact},\\
 &(iv)~ \hbox{the system}~�\eqref{55554Syst}~\hbox{below admits at least a positive constant solution},
  \end{split}
  \end{equation*}
and satisfying that, for all $\beta\in(0,1)$, there exists $C>0$ such that for all $\omega\in \mathbb{R}$ and for any positive solution $(u,v) \in C^2$ of
\begin{equation}\label{55554Syst}
\begin{cases}
\Delta_g u+ \bar{h}_\omega(u)u=u^{2^\star-1}+\omega^2(1-qv)^2 u\\
\Delta_g v+(m_1^2+q^2u^2) v= q u^2~,
\end{cases}
\end{equation}
 we have that $\|u\|_{C^{1,\beta}}+\|v\|_{C^{1,\beta}}\leq C$.
 \end{prop}
 
 The first condition on $\bar{h}$ is the most significant ; the other ones are here to ensure that 
 our result has some interest. In our definition of $\bar{h}$, we let 
\begin{equation}\label{hcheck} 
 \bar{h}_\omega(u)=\check{h}(u)+\omega^2(1-qv)^2,
 \end{equation}
 where $v=\Phi(u)$, with $\Phi$ as in \eqref{eqphi}. Thus, proving Proposition \ref{necess} amounts to prove the existence of $\check{h} : C^{1,\theta}\to C^{0,\theta}$ satisfying conditions like $(i)$-$(iv)$, and the existence of $C > 0$ 
 such that for all positive solution $u \in C^2$ of
\begin{equation}\label{55554}
\Delta_g u+ \check{h}(u)u=u^{2^\star-1}
\end{equation}
 we have that $\|u\|_{C^{1,\beta}}\leq C$. By elliptic theory, every nonnegative nonzero weak solution in $H^1$ of \eqref{55554} is actually positive and in $C^2$. As the proof of Proposition \ref{necess} below shows, we may even require that $\|\check{h}(u)\|_{C^0}=\tilde{K}$ for all $u\in C^1$, for some $\tilde{K}\gg1$ arbitrarily large and still obtain uniform bounds for the positive $C^2$-solutions of \eqref{55554}.

 \begin{proof}[Proof of Proposition \ref{necess}] 
Let $q$, $m_0$, $m_1>0$ be given. $(M,g)$ being a given closed Riemannian $n$-manifold with $n\geq 5$, we define here a map $\bar{h} : \mathbb{R}\times C^{1,\theta}\to C^{0,\theta}$ such that Proposition \ref{necess} is true. Let $\omega\in \mathbb{R}$. Note first that the function
\begin{equation*}
\Psi_0 :\left\{
\begin{array}{rcl}
\mathbb{R}^n &\to & \mathbb{R}\\
x & \mapsto & \frac{n-2}{2}u_0^2(x)+\langle x,\nabla u_0(x)\rangle_{\xi}  u_0(x),
\end{array}
\right.
\end{equation*}
is positive in $B_0\left(\sqrt{n(n-2)}\right)$, where $u_0$ is as in \eqref{u0}.
We choose any $\tilde{K}>0$ such that
\begin{equation}\label{condcontrex}
\tilde{K}>\frac{\int_{\mathbb{R}^n} u_0^2 dx}{\int_{B_0\left(\sqrt{n(n-2)}\right)}\Psi_0(x) dx}\left(-\frac{n-2}{4(n-1)}\min_{M}S_g\right).
\end{equation}
In particular, we may also ask that $\tilde{K}>\frac{n-2}{4(n-1)}\max_{M}S_g$ and we assume from now on that this is true. We pick some $\varphi\in C^\infty(\mathbb{R})$ satisfying $0\leq \varphi\leq \tilde{K}$, $\varphi(z)=\tilde{K}$ for $z\geq 0$ and $\varphi(z)=0$ for $z\leq -1$. We define, for $y\in M$ 
\begin{equation}\label{barhdef}
\begin{split}
&\check{h}(u)(y)=\varphi\left(\sqrt{\frac{n-2}{n}}|u(y)|^{\frac{n}{n-2}}-|\nabla u(y)| \right),\\
&\bar{h}_\omega(u)(y)=\check{h}(u)(y)+ \omega^2(1-q \Phi(u))^2.
 \end{split}
\end{equation}
We have that $\bar{h}_\omega(u)\geq 0$ for all $u\in C^1$ and that $\left(\tilde{K}^{\frac{n-2}{4}},\frac{q\tilde{K}^{\frac{n-2}{2}}}{m_1^2+q^2\tilde{K}^{\frac{n-2}{2}}}\right)$ is a constant positive solution of \eqref{55554Syst}. Easy computations, using the mean value theorem, give that $\check{h} : C^{1,\theta}\to C^{0,\theta}$ is continuous, and then, $\bar{h} : \mathbb{R}\times C^{1,\theta}\to C^{0,\theta}$ is also continuous by elliptic theory. To see that $\bar{h}$ satisfies $(i)$-$(iv)$ of Proposition \ref{necess}, it is sufficient to check that Step \ref{Step2lf} holds true.

\begin{Step}\label{Step2lf}
There holds $\sup_{M}\bar{h}_\omega(u)\in\left[\tilde{K},\tilde{K}+\omega^2\right]$. In particular, $\bar{h}_\omega(u)\not\equiv 0$ and $\sup_{M}\bar{h}_\omega(u)>\frac{n-2}{4(n-1)}\max_{M}S_g.$
\end{Step}

\begin{proof}[Proof of Step \ref{Step2lf}.]  As $0\leq \Phi(u)\leq \frac{1}{q}$, it is sufficient to prove that 
\begin{equation}\label{unifbd}
\sup_M \check{h}(u)=\tilde{K}~,
\end{equation}
if $u\in C^{1,\theta}$. Indeed, for any critical point $x$ of $u$, $\check{h}(u)(x)=\tilde{K}$ and \eqref{unifbd} is true. Step \ref{Step2lf} is proved.
 \end{proof}
 
  Now we prove the uniform bounds in Proposition \ref{necess}. Assume by contradiction that $((u_\alpha,v_\alpha))_\alpha$ is a blowing-up sequence of positive $C^2$ solutions of \eqref{55554Syst}, namely
\begin{equation}\label{Kstab}
\begin{cases}
\Delta_g u_\alpha+ \bar{h}_{\omega_\alpha}(u_\alpha)u_\alpha=u_\alpha^{2^\star-1}+\omega_\alpha^2(1-qv_\alpha)^2 u_\alpha\\
\Delta_g v_\alpha+(m_1^2+q^2u_\alpha^2) v_\alpha= q u_\alpha^2~,
\end{cases}
\end{equation}
where $(\omega_\alpha)_\alpha$ is a sequence of real numbers. Since $0\leq v_\alpha\leq \frac{1}{q}$, we have 
\begin{equation}\label{maxubis}
\max_M u_\alpha\to +\infty
\end{equation}
as $\alpha\to +\infty$. Once more, the first equation in \eqref{Kstab} can be written 
\begin{equation}\label{Kstabbis}
\Delta_g u_\alpha+ \check{h}(u_\alpha)u_\alpha=u_\alpha^{2^\star-1}
\end{equation}
where $\check{h}(u_\alpha)$ is as in \eqref{barhdef} and is independent of the physical parameters $m_0$, $m_1$, $q$ and $\omega_\alpha$. Replacing the system \eqref{kgstab} by \eqref{Kstab}, and thus $h_\alpha$ given in \eqref{h} by $\check{h}(u_\alpha)$, our a priori analysis of Section \ref{BlUpThryPtCtr} applies to the new sequence $((u_\alpha,v_\alpha))_\alpha$, since we only need there a $C^0$ bound on $(h_\alpha)_\alpha$. On the contrary, the arguments in Section \ref{BlUpThrySharpAsy} use the precise form of the potentials $h_\alpha$ which are replaced in our cases by the $\check{h}(u_\alpha)$'s, and need to be rewritten. First we prove that the following analogue of Proposition \eqref{p64} holds true for our sequence $((u_\alpha,v_\alpha))_\alpha$ of solutions of \eqref{Kstab}, while \eqref{prel1} is not satisfied, by $(i)$.

 \begin{prop}\label{p64bis}
 Let $(M,g)$ be a closed Riemannian $n$-manifold, $n\geq 5$, and $((u_\alpha,v_\alpha))_\alpha$ be a sequence of $C^2$ positive solutions of \eqref{Kstab} 
 such that \eqref{maxubis} holds true. 
 Let $(x_\alpha)_\alpha$ and $(\rho_\alpha)_\alpha$ be such that \eqref{hyp} holds true. Then $r_\alpha\to 0$ as $\alpha\to +\infty$,
 \begin{equation}\label{eqp642bis}
 r_\alpha= \rho_\alpha
 \end{equation}
 for all $\alpha$, and 
 \begin{equation}\label{eqp64bis}
 r_\alpha^{n-2}\mu_\alpha^{-\frac{n-2}{2}}u_\alpha\left(\exp_{x_\alpha}(r_\alpha x) \right)\to \frac{(n(n-2))^{\frac{n-2}{2}}}{|x|^{n-2}}+\mathcal{H}(x)
 \end{equation}
 in $C^1_{loc}\left(B_0(2)\backslash \{0\}\right)$ as $\alpha\to +\infty$, 
 where $\mu_\alpha$ is as in \eqref{mu}, $r_\alpha$ as in \eqref{eq40}, and $\mathcal{H}$ is a harmonic function in $B_0(2)$ which satisfies that $\mathcal{H}(0)\leq 0$.
\end{prop} 

 We first investigate the asymptotic behaviour of $\left(\check{h}(u_\alpha)\right)_\alpha$  near the blow-up point $x_\alpha$.

 \begin{Step}\label{Step1lf}
 We define $\tilde{h}_\alpha$ for $x\in \mathbb{R}^n$ by
 \begin{equation}\label{eq6lf}
 \tilde{h}_\alpha(x)=\check{h}(u_\alpha)\left(\exp_{x_\alpha}(\mu_\alpha x) \right)
 \end{equation}
 where $\mu_\alpha$ is defined in \eqref{mu} and $x_\alpha$ as in \eqref{hyp}, for all $\alpha$. Then, 
 \begin{equation}\label{eq7lf}
 \tilde{h}_\alpha\to \tilde{K}\mathds{1}_{B_0\left(\sqrt{n(n-2)}\right)} \quad \text{a.e. in }\mathbb{R}^n
 \end{equation}
 as $\alpha\to +\infty$.
 \end{Step}
 
 \begin{proof}[Proof of Step \ref{Step1lf}.] Using \eqref{eql61}, we have
\begin{equation}\label{convlocfin}
\mu_\alpha^{\frac{n}{2}}\left(\sqrt{\frac{n-2}{n}}u_\alpha^{\frac{n}{n-2}}-|\nabla u_\alpha|  \right)\left(\exp_{x_\alpha}(\mu_\alpha x) \right)\to v_0(x)\left(\sqrt{n(n-2)}-|x| \right)
\end{equation} 
in $C^0_{loc}(\mathbb{R}^n)$ as $\alpha\to +\infty$, where $v_0(x)=\frac{1}{n}\left(1+\frac{|x|^2}{n(n-2)} \right)^{-\frac{n}{2}}$ for $x\in \mathbb{R}^n$.
   Coming back to \eqref{barhdef}, we get Step \ref{Step1lf}.
 \end{proof}
 
  \begin{proof}[Proof of Proposition \ref{p64bis}] 
We are in position to prove Proposition \ref{p64bis} by adapting the proof of Proposition \ref{p64}. We let $X_\alpha$ as in \eqref{eqX}. We apply again the Pohozaev identity given in Druet and Hebey \cite{2009} to $u_\alpha$ in $B_{x_\alpha}(r_\alpha)$, we get
\begin{equation}\label{Pohobis}
\begin{split}
&\int_{B_{x_\alpha}(r_\alpha)} \check{h}(u_\alpha)\left\{u_\alpha X_\alpha(\nabla u_\alpha)+\frac{n-2}{2n}(\text{div}_g X_\alpha)u_\alpha^2 \right\} dv_g\\
&= -\frac{n-2}{4n}\int_{B_{x_\alpha}(r_\alpha)}\left( \Delta_g \text{div}_g X_\alpha \right) u_\alpha^2 dv_g+Q_{1,\alpha}-Q_{2,\alpha}+Q_{3,\alpha}~,
\end{split}
\end{equation}
where $Q_{i,\alpha}$, $i=1,2,3$, is given by \eqref{DefQIALpha}. Up to a subsequence, $x_\alpha\to \bar{x}_0$ as $\alpha\to +\infty$. For the right hand side of \eqref{Pohobis}, the second estimate in \eqref{eqcrucial2} and estimates \eqref{eql64} and \eqref{eq113} hold true using as above \eqref{eql61}, \eqref{eq41}, \eqref{eqp61} and  \eqref{estimX}. Concerning the left hand side of \eqref{Pohobis}, we get from \eqref{eql61}, \eqref{eqp61}, the dominated convergence theorem and \eqref{eq7lf} that
 \begin{equation}\label{eq9lf}
 \begin{split}
 &\int_{B_{x_\alpha}(r_\alpha)} \check{h}(u_\alpha)\left\{u_\alpha X_\alpha(\nabla u_\alpha)+\frac{n-2}{2n} \text{div}_g(X_\alpha) u_\alpha^2\right\} dv_g\\
 &= \mu_\alpha^{2}\left( \tilde{K} \int_{B_0\left(\sqrt{n(n-2)}\right)} \Psi_0 dx+o(1)\right).
 \end{split}
 \end{equation}
 Here we used that $n\geq 5$ and that $\left(\check{h}(u_\alpha)\right)_\alpha$ is bounded in $L^\infty$ by \eqref{unifbd}. Plugging these estimates in \eqref{Pohobis}, we get that
  \begin{equation}\label{eq10lf}
 \begin{split}
 Q_{1,\alpha} &=~~\mu_\alpha^2 \left(\frac{n-2}{4(n-1)} S_g(\bar{x}_0) \int_{\mathbb{R}^n} u_0^2 dx+\tilde{K}\int_{B_0(\sqrt{n(n-2)})} \Psi_0(x) dx \right) \\
 &~~+o(\mu_\alpha^2)+o(\mu_\alpha^{n-2}r_\alpha^{2-n})\\
 &\geq \eta \mu_\alpha^2 +o(\mu_\alpha^2)+o(\mu_\alpha^{n-2}r_\alpha^{2-n})
 \end{split}
 \end{equation}
 for some $\eta>0$. Indeed, by \eqref{condcontrex}, such a $\eta>0$ does exist. Since \eqref{RoughEstQ1ALpha} holds true by \eqref{eqp61} and \eqref{estimX}, we get from \eqref{eq10lf} and since $n\geq 5$ that
  \begin{equation}\label{convrbis}
 r_\alpha\to 0
 \end{equation}
 as $\alpha\to +\infty$. Defining $w_\alpha$ as in \eqref{defw1}, we get that \eqref{eq104}, \eqref{eq105} and \eqref{eq106} hold true. We then have \eqref{eqp64bis} for some $\mathcal{H}$ harmonic in $B_0(2)$. In the end, \eqref{RoughEstQ1ALpha} can be improved in \eqref{eq116} and combined with \eqref{eq10lf} to obtain 
 \begin{equation}\label{eq117bis}
\frac{1}{2} {(n-2)^2} \omega_{n-1} \Lambda \mathcal{H}(0)\leq -\eta \liminf_{\alpha\to +\infty} \left(\mu_\alpha^{4-n}r_\alpha^{n-2} \right),
\end{equation} 
  Hence, $\mathcal{H}(0)\leq 0$. At last, mimicking the end of the proof of Proposition \ref{p64} shows that \eqref{eqp642bis} holds true, which completes the proof of Proposition \ref{p64bis}.
  \end{proof}
It is now straightforward to see that the arguments developed in Section \ref{APriBdCritCase} apply to our sequence $((u_\alpha,v_\alpha))_\alpha$ of blowing-up solutions of \eqref{Kstab}, thanks to Proposition \eqref{p64bis}. In particular, we obtain a contradiction with \eqref{maxu}, as in the last part of Section \ref{APriBdCritCase}. By elliptic theory with \eqref{unifbd} and \eqref{Kstabbis}, we get that $((u_\alpha,v_\alpha))_\alpha$ is bounded in $C^{1,\beta}$ for all $\beta\in(0,1)$. This completes the proof of Proposition \ref{necess}. 
 \end{proof}
  
  Blow-up solutions for equations like \eqref{Kstabbis} can be found in Esposito, Pistoia and V\'etois \cite{EspositoPistoia}, 
  Hebey \cite{HebeyBookZurich}, and Robert and V\'etois \cite{RobVet1,RobVet2}. We refer also to 
  Druet, Hebey and V\'etois \cite{kgmp4DHV}, and Hebey and Wei \cite{HebWei1} on what concerns 
  \eqref{kg},  and to 
  Clapp, Ghimenti and Micheletti \cite{ClapGhiMic} and Ghimenti and Micheletti \cite{GhiMich} for the semiclassical setting associated to \eqref{kg}.

 \bibliographystyle{abbrv}
 \bibliography{Questions}

\begin{thebibliography}{10}

\bibitem{AmbrosettiRab}
A.~Ambrosetti and P.~H. Rabinowitz.
\newblock Dual variational methods in critical point theory and applications.
\newblock {\em J. Functional Analysis}, 14:349--381, 1973.

\bibitem{Aubin}
T.~Aubin.
\newblock \'{E}quations diff\'erentielles non lin\'eaires et probl\`eme de
  {Y}amabe concernant la courbure scalaire.
\newblock {\em J. Math. Pures Appl. (9)}, 55(3):269--296, 1976.

\bibitem{AzzPisPom}
A.~Azzollini, L.~Pisani, and A.~Pomponio.
\newblock Improved estimates and a limit case for the electrostatic
  {K}lein-{G}ordon-{M}axwell system.
\newblock {\em Proc. Roy. Soc. Edinburgh Sect. A}, 141(3):449--463, 2011.

\bibitem{AzzPom2}
A.~Azzollini and A.~Pomponio.
\newblock Ground state solutions for the nonlinear {K}lein-{G}ordon-{M}axwell
  equations.
\newblock {\em Topol. Methods Nonlinear Anal.}, 35(1):33--42, 2010.

\bibitem{BenBon}
V.~Benci and C.~Bonanno.
\newblock Solitary waves and vortices in non-{A}belian gauge theories with
  matter.
\newblock {\em Adv. Nonlinear Stud.}, 12(4):717--735, 2012.

\bibitem{BenciFortunato}
V.~Benci and D.~Fortunato.
\newblock Solitary waves of the nonlinear {K}lein-{G}ordon equation coupled
  with the {M}axwell equations.
\newblock {\em Rev. Math. Phys.}, 14(4):409--420, 2002.

\bibitem{BenFor0}
V.~Benci and D.~Fortunato.
\newblock Solitary waves of the nonlinear {K}lein-{G}ordon equation coupled
  with the {M}axwell equations.
\newblock {\em Rev. Math. Phys.}, 14(4):409--420, 2002.

\bibitem{BenFor01}
V.~Benci and D.~Fortunato.
\newblock Solitary waves in the nonlinear wave equation and in gauge theories.
\newblock {\em J. Fixed Point Theory Appl.}, 1(1):61--86, 2007.

\bibitem{BenFor2}
V.~Benci and D.~Fortunato.
\newblock Existence of hylomorphic solitary waves in {K}lein-{G}ordon and in
  {K}lein-{G}ordon-{M}axwell equations.
\newblock {\em Atti Accad. Naz. Lincei Cl. Sci. Fis. Mat. Natur. Rend. Lincei
  (9) Mat. Appl.}, 20(3):243--279, 2009.

\bibitem{BenFor4}
V.~Benci and D.~Fortunato.
\newblock Spinning {$Q$}-balls for the {K}lein-{G}ordon-{M}axwell equations.
\newblock {\em Comm. Math. Phys.}, 295(3):639--668, 2010.

\bibitem{BrezisNiremberg}
H.~Br{\'e}zis and L.~Nirenberg.
\newblock Positive solutions of nonlinear elliptic equations involving critical
  {S}obolev exponents.
\newblock {\em Comm. Pure Appl. Math.}, 36(4):437--477, 1983.

\bibitem{CaffarelliGidas}
L.~A. Caffarelli, B.~Gidas, and J.~Spruck.
\newblock Asymptotic symmetry and local behavior of semilinear elliptic
  equations with critical {S}obolev growth.
\newblock {\em Comm. Pure Appl. Math.}, 42(3):271--297, 1989.

\bibitem{Cas}
D.~Cassani.
\newblock Existence and non-existence of solitary waves for the critical
  {K}lein-{G}ordon equation coupled with {M}axwell's equations.
\newblock {\em Nonlinear Anal.}, 58(7-8):733--747, 2004.

\bibitem{ClapGhiMic}
M.~Clapp, M.~Ghimenti, and A.~M. Micheletti.
\newblock Semiclassical states for a static supercritical
  {K}lein-{G}ordon-{M}axwell-{P}roca system on a closed {R}iemannian manifold.
\newblock 2013.
\newblock Preprint.

\bibitem{AprMug2}
T.~D'Aprile and D.~Mugnai.
\newblock Non-existence results for the coupled {K}lein-{G}ordon-{M}axwell
  equations.
\newblock {\em Adv. Nonlinear Stud.}, 4(3):307--322, 2004.

\bibitem{AprMug1}
T.~D'Aprile and D.~Mugnai.
\newblock Solitary waves for nonlinear {K}lein-{G}ordon-{M}axwell and
  {S}chr\"odinger-{M}axwell equations.
\newblock {\em Proc. Roy. Soc. Edinburgh Sect. A}, 134(5):893--906, 2004.

\bibitem{AprWei1}
T.~D'Aprile and J.~Wei.
\newblock Layered solutions for a semilinear elliptic system in a ball.
\newblock {\em J. Differential Equations}, 226(1):269--294, 2006.

\bibitem{AprWei2}
T.~D'Aprile and J.~Wei.
\newblock Solutions en grappe autour des centres harmoniques d'un syst\`eme
  elliptique coupl\'e.
\newblock {\em Ann. Inst. H. Poincar\'e Anal. Non Lin\'eaire}, 24(4):605--628,
  2007.

\bibitem{AvePis}
P.~d'Avenia and L.~Pisani.
\newblock Nonlinear {K}lein-{G}ordon equations coupled with {B}orn-{I}nfeld
  type equations.
\newblock {\em Electron. J. Differential Equations}, pages No. 26, 13, 2002.

\bibitem{AvePisSic2}
P.~d'Avenia, L.~Pisani, and G.~Siciliano.
\newblock Dirichlet and {N}eumann problems for {K}lein-{G}ordon-{M}axwell
  systems.
\newblock {\em Nonlinear Anal.}, 71(12):e1985--e1995, 2009.

\bibitem{AvePisSic}
P.~d'Avenia, L.~Pisani, and G.~Siciliano.
\newblock Klein-{G}ordon-{M}axwell systems in a bounded domain.
\newblock {\em Discrete Contin. Dyn. Syst.}, 26(1):135--149, 2010.

\bibitem{Dcompactness}
O.~Druet.
\newblock Compactness for {Y}amabe metrics in low dimensions.
\newblock {\em Int. Math. Res. Not.}, 23:1143--1191, 2004.

\bibitem{2009}
O.~Druet and E.~Hebey.
\newblock Stability for strongly coupled critical elliptic systems in a fully
  inhomogeneous medium.
\newblock {\em Anal. PDE}, 2(3):305--359, 2009.

\bibitem{kgmp3}
O.~Druet and E.~Hebey.
\newblock Existence and a priori bounds for electrostatic
  {K}lein-{G}ordon-{M}axwell systems in fully inhomogeneous spaces.
\newblock {\em Commun. Contemp. Math.}, 12(5):831--869, 2010.

\bibitem{DHL}
O.~Druet, E.~Hebey, and P.~Laurain.
\newblock Stability of elliptic {PDE}s with respect to perturbations of the
  domain.
\newblock {\em J. Differential Equations}, 255(10):3703--3718, 2013.

\bibitem{DHRBlowup}
O.~Druet, E.~Hebey, and F.~Robert.
\newblock {\em Blow-up theory for elliptic {PDE}s in {R}iemannian geometry},
  volume~45 of {\em Mathematical Notes}.
\newblock Princeton University Press, Princeton, NJ, 2004.

\bibitem{kgmp4DHV}
O.~Druet, E.~Hebey, and J.~V{\'e}tois.
\newblock Stable phases for the 4-dimensional {KGMP} system.
\newblock {\em J. Reine Angew. Math.}
\newblock To appear.

\bibitem{JFA}
O.~Druet, E.~Hebey, and J.~V{\'e}tois.
\newblock Bounded stability for strongly coupled critical elliptic systems
  below the geometric threshold of the conformal {L}aplacian.
\newblock {\em J. Funct. Anal.}, 258(3):999--1059, 2010.

\bibitem{DruetLaurain}
O.~Druet and P.~Laurain.
\newblock Stability of the {P}oho\v zaev obstruction in dimension 3.
\newblock {\em J. Eur. Math. Soc. (JEMS)}, 12(5):1117--1149, 2010.

\bibitem{EspositoPistoia}
P.~Esposito, A.~Pistoia, and J.~V\'etois.
\newblock The effect of linear perturbations on the {Y}amabe problem.
\newblock {\em Math. Ann.}
\newblock to appear.

\bibitem{GeoVis}
V.~Georgiev and N.~Visciglia.
\newblock Solitary waves for {K}lein-{G}ordon-{M}axwell system with external
  {C}oulomb potential.
\newblock {\em J. Math. Pures Appl. (9)}, 84(7):957--983, 2005.

\bibitem{GhiMich}
M.~Ghimenti and A.~M. Micheletti.
\newblock Number and profile of low energy solutions for singularly perturbed
  {Klein-Gordon-Maxwell} systems on a {R}iemannian manifold.
\newblock 2013.
\newblock Preprint.

\bibitem{GidasSpruck}
B.~Gidas and J.~Spruck.
\newblock A priori bounds for positive solutions of nonlinear elliptic
  equations.
\newblock {\em Comm. Partial Differential Equations}, 6(8):883--901, 1981.

\bibitem{Gilbarg}
D.~Gilbarg and N.~S. Trudinger.
\newblock {\em Elliptic partial differential equations of second order}.
\newblock Classics in Mathematics. Springer-Verlag, Berlin, 2001.
\newblock Reprint of the 1998 edition.

\bibitem{HebeyBookZurich}
E.~Hebey.
\newblock {\em Compactness and stability for nonlinear elliptic equations}.
\newblock Zurich Lectures in Advanced Mathematics. European Mathematical
  Society.
\newblock To appear.

\bibitem{solitary}
E.~Hebey.
\newblock Solitary waves in critical abelian gauge theories.
\newblock {\em Discrete Contin. Dyn. Syst.}, 32(5):1747--1761, 2012.

\bibitem{kgmp4}
E.~Hebey and T.~T. Truong.
\newblock Static {K}lein-{G}ordon-{M}axwell-{P}roca systems in 4-dimensional
  closed manifolds.
\newblock {\em J. Reine Angew. Math.}, 667:221--248, 2012.

\bibitem{HebeyVaugon}
E.~Hebey and M.~Vaugon.
\newblock The best constant problem in the {S}obolev embedding theorem for
  complete {R}iemannian manifolds.
\newblock {\em Duke Math. J.}, 79(1):235--279, 1995.

\bibitem{HebWei1}
E.~Hebey and J.~Wei.
\newblock Resonant states for the static {K}lein-{G}ordon-{M}axwell-{P}roca
  system.
\newblock {\em Math. Res. Lett.}, 19(4):953--967, 2012.

\bibitem{schrp}
E.~Hebey and J.~Wei.
\newblock Schr\"odinger-{P}oisson systems in the 3-sphere.
\newblock {\em Calc. Var. Partial Differential Equations}, 47(1-2):25--54,
  2013.

\bibitem{LiZ}
Y.~Li and L.~Zhang.
\newblock Liouville-type theorems and {H}arnack-type inequalities for
  semilinear elliptic equations.
\newblock {\em J. Anal. Math.}, 90:27--87, 2003.

\bibitem{lizhu}
Y.~Li and M.~Zhu.
\newblock Yamabe type equations on three-dimensional {R}iemannian manifolds.
\newblock {\em Commun. Contemp. Math.}, 1(1):1--50, 1999.

\bibitem{LiZ2}
Y.~Y. Li and L.~Zhang.
\newblock A {H}arnack type inequality for the {Y}amabe equation in low
  dimensions.
\newblock {\em Calc. Var. Partial Differential Equations}, 20(2):133--151,
  2004.

\bibitem{LiZYamabe2}
Y.~Y. Li and L.~Zhang.
\newblock Compactness of solutions to the {Y}amabe problem. {II}.
\newblock {\em Calc. Var. Partial Differential Equations}, 24(2):185--237,
  2005.

\bibitem{Mug}
D.~Mugnai.
\newblock Coupled {K}lein-{G}ordon and {B}orn-{I}nfeld-type equations: looking
  for solitary waves.
\newblock {\em Proc. R. Soc. Lond. Ser. A Math. Phys. Eng. Sci.},
  460(2045):1519--1527, 2004.

\bibitem{Mug2}
D.~Mugnai.
\newblock Solitary waves in abelian gauge theories with strongly nonlinear
  potentials.
\newblock {\em Ann. Inst. H. Poincar\'e Anal. Non Lin\'eaire},
  27(4):1055--1071, 2010.

\bibitem{RobVet2}
F.~Robert and J.~V\'etois.
\newblock Examples of non-isolated blow-up for perturbations of the scalar
  curvature equation.
\newblock 2012.
\newblock Preprint.

\bibitem{RobVet1}
F.~Robert and J.~V\'etois.
\newblock A {G}eneral {T}heorem for the {C}onstruction of {B}lowing-up
  {S}olutions to {S}ome {E}lliptic {N}onlinear {E}quations with
  {L}yapunov-{S}chmidt's {F}inite-dimensional {R}eduction, {C}ocompact
  {I}mbeddings, {P}rofile {D}ecompositions, and their {A}pplications to {PDE}.
\newblock {\em Trends Math.}, pages 85--116, 2013.

\bibitem{Struwe}
M.~Struwe.
\newblock A global compactness result for elliptic boundary value problems
  involving limiting nonlinearities.
\newblock {\em Math. Z.}, 187(4):511--517, 1984.

\bibitem{Trudinger}
N.~S. Trudinger.
\newblock Remarks concerning the conformal deformation of {R}iemannian
  structures on compact manifolds.
\newblock {\em Ann. Scuola Norm. Sup. Pisa (3)}, 22:265--274, 1968.

\end{thebibliography}
\end{document}